\documentclass[a4paper, 11pt, reqno]{amsart}

\usepackage{amsmath}
\usepackage{amssymb}
\usepackage{amsthm}

\usepackage[utf8]{inputenc}
\usepackage[T1]{fontenc}
\usepackage{lmodern}
\usepackage[english]{babel}
\usepackage{mathtools}
\usepackage{esdiff}
\usepackage[normalem]{ulem}
\usepackage{cancel}
\usepackage{enumitem}
\usepackage{moreenum}
\usepackage{microtype}
\usepackage{skak}
\usepackage{bbm}
\usepackage[paper=a4paper, margin=1in]{geometry}
\usepackage{stmaryrd}
\usepackage[hidelinks]{hyperref}
\usepackage[dvipsnames]{xcolor}
\usepackage{tikz-cd}

\usepackage{dsfont}
\usepackage{multirow, makecell}
\usepackage{centernot}
\usepackage{float}
\usepackage[new]{old-arrows}
\usepackage{tikz-cd}
\usetikzlibrary{arrows}

\usepackage[numbers]{natbib}
\usepackage{color}

\makeatletter
\renewcommand{\paragraph}{\@startsection{paragraph}{4}{\z@}%
  {10\p@ \@plus 6\p@ \@minus 3\p@}%
  {-6\p@}%
  {\normalfont\itshape}}
\makeatother

\usepackage{thmtools}
\declaretheoremstyle[headfont=\normalfont\bfseries]{bfthmstyle}
\declaretheorem[numberwithin=section,style=bfthmstyle]{Theorem}
\declaretheorem[sharenumber=Theorem,style=bfthmstyle]{Lemma}
\declaretheorem[sharenumber=Theorem,style=bfthmstyle]{Proposition}
\declaretheorem[sharenumber=Theorem,style=bfthmstyle]{Corollary}

\declaretheoremstyle[headfont=\normalfont\bfseries,qed={$\diamond$}]{bfdefstyle}
\declaretheorem[sharenumber=Theorem,style=bfdefstyle]{Definition}
\declaretheorem[sharenumber=Theorem,style=bfdefstyle]{Definition-Proposition}
\declaretheorem[sharenumber=Theorem,style=bfdefstyle]{Definition-Lemma}
\declaretheorem[sharenumber=Theorem,style=bfdefstyle]{Example}
\declaretheorem[sharenumber=Theorem,style=bfdefstyle]{Notation}

\declaretheorem[numbered=no,style=bfdefstyle]{Properties}

\declaretheoremstyle[headfont=\normalfont\bfseries,qed={$\diamond$}]{bfremstyle}
\declaretheorem[sharenumber=Theorem,style=bfdefstyle]{Remark}

\newcommand{\dom}{\mathrm{dom}\,}
\newcommand{\rk}{\mathrm{rk}\,}

\newcommand{\abs}[1]{\left\vert #1 \right\vert}

\newcommand{\bigset}[2]{\left\lbrace #1 \; \middle\vert \; #2 \right\rbrace}

\newcommand{\rest}[1]{|_{#1}}

\newcommand{\dd}[1]{\tfrac{\mathrm{d}}{\mathrm{d}#1}}
\newcommand{\ddt}{\dd{t}}
\newcommand{\ddtn}[1]{\frac{\mathrm{d}^{#1}}{\mathrm{d}t^{#1}}}

\newcommand{\eanf}[1]{``#1''}

\newcommand{\ab}{\\[1.5ex]}

\newcommand{\R}{\mathbb{R}}
\newcommand{\N}{\mathbb{N}}

\newcommand{\T}{\mathbb{T}}

\newcommand{\I}{\mathbb{I}}

\newcommand{\Cinf}{{\mathcal{C}^\infty}}
\newcommand{\Cinfpw}{{\mathcal{C}^\infty_\mathrm{pw}}}

\newcommand{\pw}{_\mathrm{pw}}

\newcommand{\B}{\mathcal{B}}

\newcommand{\G}{\mathcal{G}}
\newcommand{\M}{\mathcal{M}}

\newcommand{\MD}[2]{{\M^{#1}[D]^{#2}}}
\newcommand{\MDmn}{\MD{}{m\times n}}

\newcommand{\Ca}{\mathcal{C}}

\newcommand{\op}{^{\mathrm{op}}}

\newcommand{\Vect}{\mathbf{Vect}}

\newcommand{\Int}{\mathbf{Int}}

\newcommand{\CInt}[1]{{\Ca^{#1}_\Int}}
\newcommand{\CIntp}[1]{{\Ca^{#1}_\Intp}}

\newcommand{\Intp}{{\mathbf{pInt}}}
\newcommand{\Intpp}{{\mathbf{(p)Int}}}

\newcommand{\ext}{\operatorname{ext}}
\newcommand{\transp}{^\mathsf{T}}\newcommand{\subsheaf}{\mathrel{\preccurlyeq_{\mathrm{sh}}}}

\makeatletter
\tikzset{
  edge node/.code={      \expandafter\def\expandafter\tikz@tonodes\expandafter{\tikz@tonodes #1}}}
\makeatother
\tikzset{
  subseteq/.style={
    draw=none,
    edge node={node [sloped, allow upside down, auto=false]{$\subseteq$}}},
  Subseteq/.style={
    draw=none,
    every to/.append style={
      edge node={node [sloped, allow upside down, auto=false]{$\subseteq$}}}
  }
}

\makeatletter
\DeclareRobustCommand
  \myvdots{\vbox{\baselineskip4\p@ \lineskiplimit\z@
    \hbox{.}\hbox{.}\hbox{.}}}
\makeatother

\declaretheoremstyle[
notefont=\normalfont, notebraces={}{},
headformat=\NUMBER~\NAME~\NOTE
]{nopar}

\numberwithin{equation}{section}



\begin{document}

\title{On local solutions to time-varying linear systems}
\author{Alexander Samuel Bock}
\address{Institut für Mathematik, Technische Universität Ilmenau, Weimarer Straße 25, 98693 Ilmenau, Thüringen, Germany}
\email{alexander-samuel.bock@tu-ilmenau.de}

\begin{abstract}
This paper develops an intrinsic framework for local solutions to time-varying linear systems with real meromorphic coefficients. Local solutions on compact intervals naturally form a sheaf of vector spaces. We introduce regularity triples to distinguish regularity across singularities, one-sided regularity at singularities, and regularity away from singularities. The local formulation isolates a continuation property implicit in established global frameworks and yields a simple definition of behavioural controllability. We establish an algebraically verifiable sufficient condition for continuation, prove its necessity for controllability, and derive an algebraic characterization of controllability based on the Teichmüller–Nakayama form, recovering and slightly generalizing established results.
\end{abstract}

\setcounter{tocdepth}{1}

\maketitle

\tableofcontents

\section{Introduction}\label{sec:intro}
    \subsection{Problem formulation and main contributions}\label{subsec:intro_main}

    \noindent In this paper, we develop a local framework for solutions to time-varying linear systems of the form
    \begin{equation}\label{first}
        R(\ddt) w = \sum_{i=0}^{\rho} R_i \cdot \ddtn{i} w = 0,
    \end{equation}
    where $R_i \in \M^{m \times n}$ is an $m \times n$ matrix with real meromorphic entries for $i = 0, \ldots, \rho$. Such systems have been studied by \citet{part1} and by \citet{zerz}. The ring of differential polynomials $\M[D]$ over the field of real meromorphic functions admits a rich algebraic structure that allows a factorization of differential matrix polynomials $R(D) = \sum_{i=0}^\rho R_i D^i \in \MD{m\times n}{}\cong\MD{}{m\times n}$, the Teichmüller--Nakayama form, which has proven useful for the analysis of such systems.

    \citet{part1} and \citet{zerz} consider global solutions, i.e., solutions $w \in \Cinf(\R \setminus \T, \R^n)$, where $\T \subset \R$ is a discrete, that is, a locally finite set of singularities. The reason to allow such singularities is that solutions of system~\eqref{first} may exhibit finite escape times. For example, $w(t) = t^{-1}$ solves $t \ddt w + w = 0$. In general, singularities can occur, but do not necessarily, whenever the rank of the leading coefficient drops, i.e., in a discrete set.

    The aim of the present paper is to develop a corresponding local theory on prescribed compact intervals. Throughout this paper, \eanf{local} is used in the sense that solutions are assigned to compact intervals $[a, b]$ rather than to the whole real line. This differs from local solvability near a point, where the interval of existence is not prescribed in advance.

    This paper contributes to the existing literature in the following ways.
    \renewcommand{\descriptionlabel}[1]{      \hspace\labelsep \upshape\bfseries #1    }
    \begin{description}
        \item[First,] we construct solution spaces $\mathcal{S}_R[a,b]$ on compact intervals and organize them as sheaves of vector spaces, thereby encoding restriction, gluing, and linearity in the structure of the local solution spaces. This also allows us to view $R(\ddt)$ as a sheaf morphism and the solution sheaf of~\eqref{first} as its kernel, so that algebraic tools such as the Teichmüller--Nakayama factorization can be applied in the local setting.

        The resulting framework is \emph{intrinsic} in the following sense. The solution space on $[a, b]$ is defined directly, rather than obtained extrinsically by restricting global solutions. These two approaches are not known to be equivalent in this generality, since it is not known whether compact-interval solutions satisfy the \emph{continuation property}: every solution on $[a, b]$ extends to a solution on an open neighbourhood of $[a, b]$.

        \item[Second,] we introduce \emph{regularity triples} $(L, M, N)$ to distinguish the regularity across singularities ($L$), the one-sided regularity at singularities ($M$), and the regularity away from singularities ($N$). This makes it possible to describe different types of singular behaviour and to prescribe the regularity required when local solution pieces are glued. This distinction will be used below, where controllability is formulated as a gluing property and the parameter $L$ determines the regularity imposed at the gluing points. The precise local framework is introduced in Subsection~\ref{subsec:local_framework}.

        \item[Third,] we apply the framework to a system-theoretic problem, namely \emph{behavioural controllability}. Informally, this means that any two solutions can be connected through a third solution, and the transition happens in finite time, i.e., on some finite interval. On compact intervals, this property has a particularly simple form:
        \begin{equation}\label{eq:contr_local}
            \forall a < b < c < d \; \forall f \in \mathcal{S}_R[a, b] \; \forall h \in \mathcal{S}_R[c, d] \; \exists g \in \mathcal{S}_R[b, c] \colon \quad f \cup g \cup h \in \mathcal{S}_R[a, d],
        \end{equation}
        where $f \cup g \cup h$ denotes the concatenation of $f$, $g$, and $h$.

        Regularity triples also yield regularity-dependent versions of this condition: if the chosen solution space $\mathcal{S}_R[a, d]$ consists of $(L,M,N)$-regular solutions, then the requirement $f \cup g \cup h \in \mathcal{S}_R[a,d]$ imposes only $\Ca^L$-regularity at the gluing points $b$ and $c$.

        Although the primary purpose of applying our local framework to behavioural controllability is to demonstrate its system-theoretic utility, it also allows us to slightly generalize results from the existing literature. In particular, the known algebraic characterizations of controllability in \cite[Theorem~3.2]{part1} and \cite[Theorems~6 and~7]{zerz}, which consider smooth interconnections, are recovered in Corollary~\ref{cor:main} as a consequence of Theorem~\ref{thm:main}, where $\Ca^L$-regularity at the gluing points is allowed.

        \item[Finally,] we treat the continuation property as an explicit part of the local theory. In the global framework, continuation of solutions beyond a compact subinterval is implicitly assumed by focusing on global solutions only. In the local framework, however, it becomes a separate property. The continuation property is closely related to controllability: controllability implies the continuation property, and controllability to zero is equivalent to controllability once the continuation property is assumed (Theorem~\ref{thm:contr}).

        The continuation property is not the main object of study in this paper, but it is a necessary auxiliary ingredient for the local controllability theory. Since both continuation and controllability are generally difficult to verify directly, we prove a verifiable sufficient condition in the algebraic setting: for $(L, \infty, \infty)$-regular solutions, right invertibility of $R(D)$ implies the continuation property (Lemma~\ref{lemma:cont_degree}). This prepares the proof of the main algebraic controllability criterion (Theorem~\ref{thm:main}).
    \end{description}

    The main contribution of this paper is therefore the intrinsic compact-interval solution framework, including regularity triples and the separate treatment of continuation. Controllability serves primarily as a system-theoretic application of this framework. A detailed comparison with existing approaches is given in Subsection~\ref{subsec:global_comp}.

    \subsection{Local framework and regularity triples}\label{subsec:local_framework}

    \noindent In this subsection, the solution space $\mathcal{S}_R[a, b]$ is defined, its sheaf structure is investigated, and the solution space is refined using the definition of regularity triples.

    Consider the real vector space of partial real functions on $[a, b]$ with finitely many singularities
    \[
        \Ca^{-1}[a, b] \coloneqq \bigset{f\colon [a, b] \setminus E \to \R}{E \subseteq [a, b] \text{ finite}} / {\sim}, \qquad -\infty < a < b < \infty,
    \]
    where
    \[
        f \sim g \quad :\Longleftrightarrow \quad \exists E \subseteq [a, b] \text{ finite } \forall t \in \left(\dom f \cap \dom g\right) \setminus E \colon f(t) = g(t).
    \]
    As in the Lebesgue space $L^p(\R)$, equivalence classes are identified with their representative functions. The subspace of piecewise $N$-times continuously differentiable functions on $[a, b]$ is defined as
    \[
        \Ca^N\pw [a, b] \coloneqq \bigset{w \in \Ca^{-1}[a, b]}{\exists f\in w \colon f \in \Ca^N(\dom f, \R)}.
    \]
    The restriction map
    \[
        \M \to \Ca^N\pw [a, b],
        \quad
        f \mapsto f\rest{[a, b]} \coloneqq [f\rest{[a, b] \cap \dom f}]_{\sim}
    \]
    is well-defined and linear. For \(w \in \left(\Ca^N\pw [a, b]\right)^n\), we therefore simplify notation and write \(R_i \ddtn{i} w\) instead of \(R_i\rest{[a, b]} \ddtn{i} w\). For \(N \geq \rho = \deg R(D)\), we may define
    \[
        R(\ddt) w = \sum_{i=0}^{\rho} R_i \cdot \ddtn{i} w.
    \]
    Note that, here and throughout, differentiation of elements of $\Ca^N\pw[a,b]$ is understood piecewise rather than distributionally.

    The set of solutions on $[a, b]$ is given by
    \[
        \mathcal{S}_R[a, b] \coloneqq \bigset{w \in \left(\Ca^N\pw [a, b]\right)^n}{R(\ddt) w = 0}.
    \]
    Since system~\eqref{first} is linear, $\mathcal{S}_R[a, b]$ is a subspace of $\left(\Ca^N\pw [a, b]\right)^n$.

    In many cases, if not stated otherwise, we set $N = \rho$.

    \begin{Example}\label{ex:first_example}
        Consider the equation
        \[
            (t^4, 0) \, \ddtn{2} w + (4t^3, 0) \, \ddtn{} w + (t^2, -1) \, w = 0.
        \]
        The corresponding differential matrix polynomial is
        \[
            r(D) = (t^4, 0) \, D^2 + (4t^3, 0) \, D + (t^2, -1) \in \MD{1\times 2}{}.
        \]
        One solution $w \in \mathcal{S}_r[-1, 1] = \bigset{v \in \Ca^2\pw [-1, 1]}{r(\ddt) v = 0}$ is given by the function
        \[
            f(t) = \left( \abs{t}^{3/2}, \frac{31}{4} t^2 \abs{t}^{3/2} \right) \transp, \qquad t \in [-1, 1],
        \]
        via $w \coloneqq [f]_{\sim} \in \mathcal{S}_r[-1, 1]$. To simplify our notation, we will not distinguish between $f$ and $w$ and just write $f \in \mathcal{S}_r[-1, 1]$. Note that $f \notin \left(\Ca^2[-1, 1]\right)^2$ because $f$ has a singularity at $t = 0$, where the regularity drops to $\Ca^1$.
        \qedhere

    \end{Example}

    The assignment $\mathcal{S}_R \colon [a, b] \mapsto \mathcal{S}_R[a, b]$ with $a < b$ satisfies the following properties.

    \begin{Properties}\label{prop}
        \begin{enumerate}[label=(P\arabic*)]
            \item \label{prop:i} Restrictions of solutions are again solutions, i.e., there are well-defined restriction maps
            \[
                \mathcal{S}_R[a, d] \to \mathcal{S}_R[b, c], \quad w = [g]_{\sim} \mapsto w\rest{[b, c]} \coloneqq [g\rest{\dom g \cap [b, c]}]_{\sim},
            \]
            for $a \leq b < c \leq d$.
            \item \label{prop:ii} Restrictions compose, i.e., for $a \leq b \leq c < d \leq e \leq f$ and for all $w \in \mathcal{S}_R[a, f]$
            \[
                w\rest{[c, d]} = w\rest{[b, e]}\rest{[c, d]}.
            \]
            \item \label{prop:iii} \emph{Compatible solutions} are uniquely glueable, i.e., if $f \in \mathcal{S}_R[a, c]$ and $g \in \mathcal{S}_R[b, d]$ with $a \leq b < c \leq d$ are equal on $[b, c]$, i.e.,
            \[
                f\rest{[b, c]} = g\rest{[b, c]},
            \]
            then there exists a unique $h \in \mathcal{S}_R[a, d]$ with $h\rest{[a, c]} = f$ and $h\rest{[b, d]} = g$.
            \item \label{prop:iv} $\mathcal{S}_R[a, b]$ is a real vector space.
            \item \label{prop:v} Every restriction map is linear. \qedhere
        \end{enumerate}
    \end{Properties}

    \hyperref[prop]{Properties}~\ref{prop:i}--\ref{prop:iii} reflect the fundamental behaviour expected of solutions to differential equations. \hyperref[prop]{Properties}~\ref{prop:iv}--\ref{prop:v} express additional structure specific to linear systems.

    In modern mathematical language, the assignment $[a, b]\mapsto\mathcal{S}_R[a, b]$ is called a \emph{presheaf} if \ref{prop:i}--\ref{prop:ii} hold, and a \emph{sheaf} if \ref{prop:i}--\ref{prop:iii} hold; if in addition~\ref{prop:iv}--\ref{prop:v} are satisfied, one speaks of a \emph{presheaf} or \emph{sheaf of vector spaces}.

    \begin{Example}\label{ex:presheaf_not_sheaf}
        Consider the system $r(\ddt)w=\ddt w=0$ and the global piecewise smooth solution given by $w(t)=1$ if $\abs{t}\geq 2$ and $w(t)=0$ otherwise. Define
        \[
            \mathcal F[a,b]\coloneqq \operatorname{span}_{\R}\{w\rest{[a,b]}\}\subseteq\mathcal S_r[a,b]\qquad\text{for $a<b$}.
        \]
        Then $\mathcal F$ is a presheaf of vector spaces but not a sheaf, i.e., it satisfies all \hyperref[prop]{Properties}~\ref{prop:i}--\ref{prop:v} except~\ref{prop:iii}: $f\coloneqq w\rest{[-3,1]}$ and $g\coloneqq 2w\rest{[-1,3]}$ agree on $[-1,1]$, but cannot be glued in $\mathcal F[-3,3]$. Nevertheless, since $\mathcal S_r$ is a sheaf, their concatenation $f\cup g\in\mathcal S_r[-3,3]$ is again a solution.
    \end{Example}

    This motivates a minimal structural requirement for a local solution concept in the linear setting: the assignment of solution spaces should form a sheaf of vector spaces. This requirement alone is not sufficient for the system-theoretic questions considered below. Indeed, any three solutions in $\mathcal{S}_R[a,b]$, $\mathcal{S}_R[b,c]$, and $\mathcal{S}_R[c,d]$ can be glued while allowing singularities at $b$ and $c$. Consequently, every system~\eqref{first} would be controllable under definition~\eqref{eq:contr_local}.

    Before proposing a definition of a solution space suited for controllability, note that $\Ca^N[a, b]$ can be viewed as a subspace of $\Ca^N\pw[a, b]$ by identifying $f \in \Ca^N[a, b]$ with $[f]_{\sim} \in \Ca^N\pw[a, b]$ (as done in Example~\ref{ex:first_example}).

    Using this convention, we define, for $-1 \leq L \leq M \leq N \in \N_0 \cup \{\infty\}$,
    \[
        \Ca^{L, M, N}[a, b] \coloneqq \bigset{f \in \Ca^L[a, b] \cap \Ca^N\pw[a, b]}{\begin{array}{l}
            \exists E \subseteq [a, b] \text{ finite } \forall (x, y) \subseteq (a, b)\setminus E \colon \\
             f\rest{[x, y]} \in \Ca^M[x, y]
        \end{array}}.
    \]
    The triple $(L, M, N)$ is referred to as a \emph{regularity triple}. This induces the solution space
    \[
        \mathcal{S}_R^{L, M, N}[a, b] \coloneqq \bigset{w \in \left(\Ca^{L, M, N}[a, b]\right)^n}{R(\ddt) w = 0}, \qquad N \geq \rho = \deg R(D).
    \]
    It is straightforward to show that $\mathcal{S}_R^{L, M, N}$ also satisfies \hyperref[prop]{Properties}~\ref{prop:i}--\ref{prop:v}, i.e., it is a sheaf of vector spaces. Unlike $\mathcal{S}_R$, however, its controllability in the sense of~\eqref{eq:contr_local} is nontrivial, and its relation to the existing notions of controllability is discussed in Subsection~\ref{subsec:global_comp}.

    We close this subsection with a few remarks and examples concerning the choice of the solution sheaf $\Ca^{L, M, N}$.

    \begin{Remark}
        For $f \in \Ca^{L, M, N}[a, b]$, $L$ specifies the overall regularity of $f$, or equivalently, the minimal smoothness required at singularities. $M$ determines the regularity of each piece $f\rest{[x, y]}$, i.e., the number of derivatives that exist at subinterval boundaries. $N$ determines the regularity within each piece and thus the smoothness almost everywhere on $[a, b]$.

        Here, the value $-1$ means that no corresponding regularity condition is imposed. Derivatives are understood piecewise rather than distributionally; hence jumps do not produce Dirac impulses, and distributional or switched-system solution concepts are not considered.

        For our purposes, $L$ controls the strength of the compatibility notion, or equivalently, the regularity of the control, while $M$ is needed for technical reasons, as we will later take $M = \infty$ to ensure continuability of solutions.

        As usual, the greater the parameters $L, M, N$, the more regular functions are required to be, and the smaller the corresponding space $\Ca^{L, M, N}[a, b]$. That is,
        \[
            \Ca^{L, M, N}[a, b] \supseteq \Ca^{L', M', N'}[a, b] \quad \Longleftrightarrow \quad L \leq L'\ \land\  M \leq M'\ \land\ N \leq N'.
        \]
        Note that $\Ca^{N, N, N}[a, b] = \Ca^{N}[a, b]$ and $\Ca^{-1, -1, N}[a, b] = \Ca^N\pw[a, b]$.
    \end{Remark}

    \begin{Example}
        We give a few examples of functions $f \in \Cinfpw[-1, 1]$ with a single singularity at $t = 0$ and their regularity triples, i.e., the maximal parameters $L, M, N$ such that $f \in \Ca^{L, M, N}[-1, 1]$. Thus, $N = \infty$ in every example, and $L$ and $M$ are solely determined by the regularity of the singularity at $t = 0$. We have
        \begin{enumerate}[label=(\roman*{})]

            \item $\operatorname{sgn} t \rest{[-1, 1]}\in \Ca^{-1, \infty, \infty}[-1, 1]$, the jump discontinuity at $t = 0$ implies $L = -1$, and the halves are $\Ca^\infty[-1, 0]$ and $\Ca^\infty[0, 1]$, respectively ($M = \infty$); \label{ex:reg_i}
            \item $\abs{t} \rest{[-1, 1]} \in \Ca^{0, \infty, \infty}[-1, 1]$, the function is continuous ($L=0$), and the halves are $\Ca^\infty[-1, 0]$ and $\Ca^\infty[0, 1]$, respectively ($M = \infty$);
            \item $\abs{t}^{1/2} \rest{[-1, 1]} \in \Ca^{0, 0, \infty}[-1, 1]$, the function is continuous ($L=0$), and the halves are $\Ca^0[-1, 0]$ and $\Ca^0[0, 1]$, respectively ($M = 0$);\label{ex:reg_iii}
            \item $t^{-1}\rest{[-1, 1]}\in \Ca^{-1, -1, \infty}[-1, 1]$, the pole at $t = 0$ implies $L = M = -1$;
            \item $\left(\operatorname{sgn} t + \abs{t}^{1/2}\right)\rest{[-1, 1]}\in \Ca^{-1, 0, \infty}[-1, 1]$, use \ref{ex:reg_i} and \ref{ex:reg_iii}. \qedhere
        \end{enumerate}
    \end{Example}

\subsection{Comparison to related work}\label{subsec:global_comp}

    \noindent In Subsection~\ref{subsec:intro_main}, we outlined the main contributions of this paper to the existing literature. Having introduced the local framework in Subsection~\ref{subsec:local_framework}, we are now in a position to compare it with existing approaches in greater detail.

    \paragraph{Explicit frameworks for system~\eqref{first}}

    We first compare our framework with two closely related approaches: the global framework used in \cite{zerz} and the local framework used in \cite{part1}. Both works consider system~\eqref{first} and define the function space
    \[
        S \coloneqq \bigset{f \in \Ca^\infty(\R\setminus\T, \R^n)}{\T \subset \R \text{ discrete}} \big/ {\sim},
    \]
    where $f \sim g$ if and only if $f = g$ $\lambda$-a.e. This identification turns $S$ into a real vector space, although it is left implicit in \cite{part1}. The corresponding global solution space of system~\eqref{first} is
    \[
        \ker_S R(\ddt) = \bigset{w \in S}{R(\ddt)w = 0}.
    \]
    Moreover, \citet{part1} also consider a local solution space at $t \in \R$, defined by
    \[
        S_t = \bigset{f \in \Ca^\infty(\I, \R^n)}{\I \ni t \text{ an open interval}},
    \]
    together with
    \[
        \ker_t R(\ddt) = \bigset{w \in S_t}{R(\ddt)w = 0}.
    \]

    Both works study behavioural controllability, using slightly different but equivalent formulations.

    In \cite{part1}, behavioural controllability is defined locally at $t$ for the set $\ker_t R(\ddt)$; see \cite[Definition~3.1]{part1}. Roughly speaking, this means that any two smooth solutions near $t$ admit a smooth interconnection near $t$. We do not discuss the methods of \cite{part1} in detail, because there are two substantial issues in the formulation of the local solution space.

    First, the set $\ker_t R(\ddt)$, equipped with the structure described in \cite[Remark~1~(ii)]{part1}, is not a real vector space as stated. The claimed vector space structure would require an additional identification of functions that agree on a neighbourhood of $t$. Such an identification would lead to the sheaf-theoretic notion of a stalk at $t$. However, this is not the object used in their subsequent arguments. Nevertheless, the vector space property is used essentially, for instance in the equivalence between controllability and controllability to zero; see \cite[Remark~4~(ii)]{part1}.

    Second, \citet{part1} suggest that $\ker_S R(\ddt)$ is equal to $\bigcup_{t\in\R} \ker_t R(\ddt)$ and define global controllability for the latter set. This equality is false as stated. Hence their results should be understood as local rather than global.

    Altogether, it is unclear to what extent the results of \cite{part1} hold in the stated form. We expect, however, that many of them can be recovered after replacing the local solution spaces by a more rigorous local construction. Indeed, much of the proof strategy for Theorem~\ref{thm:main} is taken from the proof of \cite[Theorem~3.2]{part1}.

    In \cite[Definition~3]{zerz}, controllability is formulated for global solutions: $\ker_S R(\ddt)$ is \emph{controllable} if, for all solutions $w_1, w_2 \in \ker_S R(\ddt)$ and for almost every $t_0 \in \R$, there exist a solution $w \in \ker_S R(\ddt)$, an open interval $\I \ni t_0$, and $\tau > 0$ with $t_0+\tau \in \I$ such that $w_1$, $w_2$, and $w$ are smooth on $\I$ and
    \begin{equation}\label{eq:contr_zerz}
        w(t) = \begin{cases}
            w_1(t), & t < t_0,\\
            w_2(t), & t > t_0+\tau,
        \end{cases}
    \end{equation}
    for all $t \in \I$.

    Using methods from algebraic analysis, \citet{zerz} obtains an algebraic characterization of controllability; see \cite[Theorems~6 and~7]{zerz}. This characterization is equivalent to \cite[Theorem~3.2]{part1}. In our local framework, the result reads as follows:
    \begin{quote}
        \textbf{Corollary~\ref{cor:main}.} If $R(D) \in \MDmn$ has full row rank, then $\mathcal{S}_R^{\infty,\infty,\infty}$ is controllable in the sense of property~\eqref{eq:contr_local} if and only if $R(D)$ is right invertible.
    \end{quote}
    This already gives a cleaner formulation of the controllability criterion: $\mathcal{S}_R^{\infty,\infty,\infty}$ consists only of smooth local solutions, whereas $\ker_S R(\ddt)$ allows singularities, and property~\eqref{eq:contr_local} is conceptually simpler than property~\eqref{eq:contr_zerz}. Using regularity triples, we further generalize this statement.

    \begin{quote}
        \textbf{Theorem~\ref{thm:main}.} If $R(D) \in \MDmn$ has full row rank and $L$ is larger than some lower bound depending on $R(D)$, then $\mathcal{S}_R^{L, \infty, \infty}$ is controllable in the sense of property~\eqref{eq:contr_local} if and only if $R(D)$ is right invertible.
    \end{quote}

    Note that \cite{part1} and \cite{zerz}, in addition to controllability, also consider, among other things, image representations, autonomous behaviour, observability, and the elimination of latent variables. These topics are not pursued in the present paper. Our aim is instead to construct a rigorous local framework for systems with meromorphic coefficients. Controllability is used here primarily as a system-theoretic test case demonstrating that this local framework recovers and slightly generalizes the criterion above.

    \paragraph{An abstract sheaf-theoretic framework}

    Another sheaf-theoretic framework for dynamical systems is developed in~\cite{Schultz2020}. There, state, input, and output trajectories are represented by sheaves on a category of compact intervals, and abstract systems are modelled as \emph{machines}. Their interval category uses the standard intervals $[0,t]$, $t\geq 0$. This is less suitable for time-varying systems, whose solution spaces depend on absolute time.

    The main difference, however, is the level of generality. While \cite{Schultz2020} works with abstract interval sheaves in order to study arbitrary open systems and their interconnections, we consider particular sheaves of functions and solution subsheaves defined by time-varying linear differential equations with meromorphic coefficients. Our framework is therefore more explicit and tailored to the analytic questions considered here.

    A notion related to the continuation property is their concept of \emph{totality}; see \cite[Definition~4.2.2]{Schultz2020}. Disregarding the difference between the interval categories, the restricted sheaf morphism (a rigorous definition is given in Subsection~\ref{subsec:sheaf_morphisms})
    \[
        R(\ddt)\colon \mathcal{S}_R^{L, M, N}\longrightarrow 0, \qquad \text{for $N \geq \rho = \deg R(D)$}
    \]
    may be defined as the input map of a machine. Totality means that every solution section admits a right extension of arbitrary length; our continuation property, on the other hand, requires only a local but two-sided extension. Note that these differences result from the different purposes of the two notions. Roughly speaking, totality requires that admissible future inputs allow for future state evolutions, whereas our continuation property is conceptually closer to requiring that the maximal interval of existence of solutions is open.

    \paragraph{Continuation of local solutions}

    The continuation property considered here is more directly related to the classical asymptotic theory developed by \citet{Wasow1987}. The analysis concerns linear differential equations in the complex domain near regular and irregular singular points and establishes, under analytic assumptions, the existence of holomorphic solutions on suitable sectors that admit prescribed formal power series as asymptotic expansions. These results may be relevant to the continuation property considered here, but their translation to our real setting is not immediate. In particular, they concern sectorial analytic solutions with prescribed asymptotic behaviour rather than the continuation of an arbitrary solution given on a compact interval. We have therefore not been able to determine whether they imply the continuation property in our setting.

    Since the continuation property is necessary for controllability, we identify conditions under which it can be established, in particular Lemma~\ref{lemma:cont_degree}.

    \paragraph{Solution spaces with singularities}

    $\Ca^{L, M, N}$-functions, as well as global or slightly modified variants, appear in the literature, although typically only for specific choices of the regularity triple $(L, M, N)$. Table~\ref{tab:regularity_examples} lists a non-exhaustive selection of existing works where $\Ca^{L, M, N}$ or closely related function spaces are utilized.

        \begin{table}[H]
            \centering
            \setcellgapes{4pt}
            \makegapedcells

            \begin{tabular}{|p{0.8cm}p{0.8cm}p{0.8cm}|l|l|}
                \hline
                \centering $L$ & \centering $M$ & \centering $N$ & Context & Appearance \\
                \hline \hline
                \centering $-1$ & \centering $-1$ & \centering $\infty$ & system~\eqref{first} &
                \makecell[l]{in \cite{part1} as \eanf{piecewise smooth, $\Cinfpw$},\\
                in \cite{zerz} as \eanf{$\mathcal A$}} \\
                \hline
                \centering $-1$ & \centering $\infty$ & \centering $\infty$ & distributional DAEs &
                \makecell[l]{in \cite[Definition~2.2.12]{trennphd} as\\
                \eanf{piecewise smooth, $\Cinfpw$}} \\
                \hline
                \centering $0$ & \centering $0$ & \centering $\infty$ & PDEs &
                \makecell[l]{in \cite{Guelmame_2022} as \eanf{weakly singular}} \\
                \hline
                \centering $0$ & \centering $k$ & \centering $k$ & nonsmooth analysis &
                \makecell[l]{similarly in \cite{Scholtes2012} as \eanf{$PC^k$}} \\
                \hline
                \centering $0$ & \centering $\infty$ & \centering $\infty$ & distributional DAEs &
                \makecell[l]{in \cite[Theorem~4.2.15]{trennphd}} \\
                \hline
                \centering $k$ & \centering $k+1$ & \centering $k+1$ & approximation theory &
                \makecell[l]{in \cite{levin:approximating_pw_smooth} as \eanf{$PC^{k+1}$}} \\
                \hline
            \end{tabular}
            \caption{Examples of regularity triples $(L, M, N)$ and their appearance in the literature.}
            \label{tab:regularity_examples}
        \end{table}

        With the exception of \cite{Scholtes2012} and \cite{levin:approximating_pw_smooth}, these works consistently consider global function spaces and introduce them directly as solution spaces for differential equations.

        Other notions of solutions with singularities appear in the literature, most notably Carathéodory and Filippov solutions for ordinary differential equations with discontinuous right-hand sides, see~\cite{filippov1988}. Both are absolutely continuous and thus conceptually comparable to $\Ca^{0, 0, 1}$, but admit singularities on null sets rather than finite sets as in the present framework.

    \subsection{Plan of the paper}\label{subsec:plan}

    \noindent The remainder of this paper is organized as follows.

    Section~\ref{sec:loc_sol} motivates and completes the formulation of our local framework. In Subsection~\ref{subsec:examples}, we examine two mechanical systems described by differential equations with input. The aim is to demonstrate both the strengths and the limitations of the sheaf $\mathcal{C}^{L, M, N}$. The subsection also serves to familiarize the reader with the three-parameter solution spaces before the abstract development begins. Its contents are not needed for the subsequent results, so readers who wish to proceed directly to the abstract part may skip it and continue with the following subsection.

    In Subsection~\ref{subsec:Int_comp}, we extend $\mathcal{C}^{L, M, N}$ defined on the category $\Intp$ of proper compact intervals to the category $\Int$ of all compact intervals including functions on degenerate intervals $[b, b]$. This addresses the issue of \emph{compatibility at a point}: in $\mathcal{C}^{L, M, N}$, compatibility of two functions $f$ and $g$ defined on $[a, b]$ and $[b, c]$ cannot be decided using Property~\ref{prop:iii}. Instead, $f$ and $g$ together form a function in $\mathcal{C}^{L, M, N}[a, c]$ if and only if the connection at $b$ is of regularity $\Ca^L$; we say $f$ and $g$ are \emph{$\Ca^L$-compatible (at $b$)}. By extending $\mathcal{C}^{L, M, N}$ to $\Int$, compatibility at a point~$b$ can be expressed as an intrinsic property of the sheaf.

    In Section~\ref{sec:contr}, we apply the local framework to systems theory, with a focus on controllability. Subsection~\ref{subsec:contr_main} introduces the notion of controllability for presheaves of vector spaces on $\Int$ and contains controllability results for sheaves of a specific regularity, which we call \emph{local behaviours}.

    In Subsection~\ref{subsec:sheaf_morphisms}, we interpret the differential operator $R(\ddt)$ as a morphism of sheaves. This setup allows one to view the corresponding local behaviour of solutions as the kernel of this sheaf morphism. In particular, this viewpoint prepares us to derive an algebraic characterization of controllability for systems with meromorphic coefficients (Theorem~\ref{thm:main}) in Subsection~\ref{subsec:tnnf}, utilizing the Teichmüller--Nakayama form.

    The \hyperref[app]{appendix} provides additional categorical background that justifies the use of the term \eanf{sheaf} in the main text. In particular, it verifies that the definition of sheaves using \hyperref[prop]{Properties}~\ref{prop:i}--\ref{prop:v} is consistent with the standard sheaf-theoretic framework. It also contains the proof of Theorem~\ref{thm:Int&Intp} that relies on more categorical language and was therefore moved out of the main exposition.

    In order to keep the paper accessible to readers unfamiliar with abstract sheaf theory, we use concrete definitions throughout (for example, defining sheaves on the category $\Intp$ via \hyperref[prop]{Properties}~\ref{prop:i}--\ref{prop:v}). In fact, the text is intended to remain readable without prior knowledge of categories, functors, or sheaves, although familiarity with these notions is helpful. In most places, the term \eanf{category} may simply be replaced by \eanf{set}; the term \eanf{functor} is not used; and the term \eanf{sheaf} is always introduced through explicit definitions.

\section{Properties of the local framework}\label{sec:loc_sol}

    \subsection{Physical interpretation and modelling}\label{subsec:examples}

    \noindent While the primary application of our framework is behavioural controllability (discussed in Section~\ref{sec:contr}), we first introduce two concepts that demonstrate the framework's broader utility: the behavioural approach presented by \citet{book:willems} and the concept of reinitialization. The behavioural approach is of particular importance, as it provides the conceptual foundation for the notion of behavioural controllability. This subsection also serves to familiarize the reader with the solution spaces $\mathcal{C}^{L, M, N}[a, b]$, and we therefore recommend reading it before proceeding to the abstract framework. Its contents are not required for the subsequent developments, however, so readers wishing to proceed directly may continue with Subsection~\ref{subsec:Int_comp}.

    In its simplest case, the behavioural approach \citep{book:willems} transforms an inhomogeneous linear differential equation of the form
    \begin{equation} \label{eq:inhom}
        \sum_{i = 0}^N A_i x^{(i)} = G u
    \end{equation}
    into a homogeneous linear differential equation of the form
    \begin{equation} \label{eq:hom}
        \sum_{i = 0}^N R_i w^{(i)} = 0,
    \end{equation}
    by setting $w \coloneqq (x\transp, u\transp)\transp$, $R_i \coloneqq (A_i, 0)$ for $i = 1, \dots, N$, and $R_0 \coloneqq (A_0, -G)$. In the classical formulation~\eqref{eq:inhom}, $x$ and $u$ play very different roles, since only $x$ is acted on by differential operators. In contrast, the behavioural approach treats $u$ as part of the solution~$(x,u)$.

    The primary advantage of this formulation is that equation~\eqref{eq:hom} is algebraically simpler than equation~\eqref{eq:inhom}; for instance, the solution space of the homogeneous equation constitutes a vector space. A potential disadvantage, however, is that this representation requires the input $u$ to be~$N$ times differentiable for equation~\eqref{eq:hom} to remain well-defined. In our framework, we adopt the compromise that $u$ must be $N$ times differentiable except at a discrete set of singularities, such as jumps. We demonstrate in Example~\ref{ex:physics}~\ref{ex:physics:(i)} that our framework enables us to specify the exact class of irregularities permitted in the input, ensuring that the solution space remains expressive enough to align with physically sensible behaviours.

    The second concept, reinitialization, concerns the process of halting a dynamical system at time $t$ and restarting it using the values $w(t), \dot{w}(t), \dots, w^{(N-1)}(t)$ as new initial conditions. The Picard--Lindelöf theorem ensures that, under suitable regularity assumptions, a reinitialized solution follows the same trajectory as the original one, that is, the system behaves as if no external interaction had occurred. Specifically, if two solutions are $\Ca^{N-1}$-compatible in $t$, then they are even $\Ca^N$-compatible, meaning no drop in regularity can occur and the future is uniquely determined. In these scenarios, all solutions are $\Ca^N$ and the added expressiveness of our framework via the regularity triple $(L, M, N)$ is unnecessary.

    However, our framework also applies to situations where Picard--Lindelöf fails. This allows one to parametrize non-unique continuations that are physically sensible, as demonstrated in Example~\ref{ex:physics}~\ref{ex:physics:(ii)}.

    \begin{Example}\label{ex:physics}
        \begin{enumerate}[label=(\roman*{})]
            \item \label{ex:physics:(i)} Consider a particle of mass $m$ at position $x(t)$ on a finite interval $[a, b]$, $a< b$, subject to an external force $u(t)$. Newton's second law yields the inhomogeneous differential equation
            \[
                m\, \ddot{x}(t) = u(t).
            \]
            Defining $w(t) \coloneqq (x(t), u(t))\transp \in \R^2$ leads to the equivalent homogeneous equation
            \[
                (m, 0) \, \ddot{w}(t) + (0, 0) \, \dot{w}(t) + (0, -1) \, w(t) = 0.
            \]
            Some regularity of the input must be assumed, as is typical in the behavioural approach. There are, however, situations where an irregular input must be modelled. For example, one may consider a force $u$ that is bounded and piecewise $\Ca^2$ with finitely many jump discontinuities in $[a, b]$. In our setting, these properties are captured by the slightly stronger assumption $u \in \Ca^{-1, 0, 2}[a, b]$. Integration yields $x\in \Ca^{-1, 2, 4}[a, b]$.

            If one additionally imposes the physically natural assumption $x \in \Ca^1[a, b]$, then
            \[
                w \in \Ca^{1, 2, 4}[a, b] \times \Ca^{-1, 0, 2}[a, b].
            \]
            This componentwise description requires six regularity parameters, and in general $3n$ parameters. To simplify the framework, we instead adopt the minimal common regularity and require only
            \[
                w \in \left(\Ca^{-1, 0, 2}[a, b]\right)^2.
            \]

            \item \label{ex:physics:(ii)}

            Consider a pendulum of mass $m = 1$ with a time-dependent length $\ell(t) \geq 0$, subject to an external torque $u(t)$. Such a setup can indeed be realized by a slight modification of one of the experimental setups reviewed in \cite[Section~2]{Yakubu2022}. For small angles $x(t)$, the system is modelled by the time-varying linear differential equation
            \[
                \ell(t)^2 \, \ddot{x}(t) + 2 \ell(t)\dot{\ell}(t) \, \dot{x}(t) + \ell(t) \, x(t) = u(t),
            \]
            where the gravitational constant is normalized to $g = 1$. The system is observed on the interval $[-\varepsilon, \varepsilon]$ for sufficiently small $\varepsilon > 0$.

            Set $\ell(t)\coloneqq t^2\in\Ca^\infty[-\varepsilon,\varepsilon]$ and
            \[
                u(t)\coloneqq\frac{31}{4}t^2|t|^{3/2}\in\Ca^{3,3,\infty}[-\varepsilon,\varepsilon].
            \]
            Since equality is understood modulo finitely many points, division by $t^2$ leaves the solution set unchanged. Thus, the equation reduces to
            \[
                t^2\,\ddot{x}(t)+4t\,\dot{x}(t)+x(t)=\frac{31}{4}|t|^{3/2}.
            \]
            The function
            \[
                x(t) = |t|^{3/2} \in \Ca^{1, 1, \infty}[-\varepsilon, \varepsilon]
            \]
            is a solution to this equation, and the behaviour
            \[
                w \coloneqq (x, u)\transp \in \left(\Ca^{1, 1, \infty}[-\varepsilon, \varepsilon]\right)^2
            \]
            satisfies
            \[
                (t^4, 0)  \, \ddot w(t) + (4t^3, 0)  \, \dot w(t) + (t^2, -1)  \, w(t) = 0.
            \]
            The Picard--Lindelöf assumptions fail at $t = 0$ as the leading coefficient $(t^4, 0)$ drops in rank. As a result, reinitializing the system with the initial values $w(0) = 0$ and $\dot{w}(0) = 0$ does not yield a unique continuation. Indeed, for every $c \in \mathbb{R}$, $w\rest{[-\varepsilon, 0]}$ and $c \cdot w\rest{[0, \varepsilon]}$ are $\Ca^1$-compatible, and their concatenation lies in $\left(\Ca^{1, 1, \infty}[-\varepsilon, \varepsilon]\right)^2$. The present framework allows one to choose $L = 1$, admitting physically sensible solutions with singularities while excluding jumps.

            This non-uniqueness is not an artifact of linearization. For $\ell(t)=t^2$, the nonlinear model
            \[
                 \ell(t)^2 \, \ddot{x}(t) + 2 \ell(t)\dot{\ell}(t) \, \dot{x}(t) + \ell(t) \, \sin(x(t)) = u(t),
            \]
            admits the family
            \[
                w(t) = (x(t), u(t))\transp = \left(c \abs{t}^{3/2}, c \frac{27}{4} t^2 \abs{t}^{3/2} + t^2 \sin\big(c\abs{t}^{3/2}\big) \right)\transp, \qquad c \in \R,
            \]
            of solutions with the same regularity as above and initial values $w(0) = 0$ and $\dot{w}(0) = 0$. \qedhere

        \end{enumerate}
    \end{Example}

    \subsection{Formalizing compatibility}\label{subsec:Int_comp}
    \noindent The \hyperref[prop]{Properties}~\ref{prop:i}--\ref{prop:v} listed in the \hyperref[sec:intro]{introduction} characterize the structure of solutions on proper compact intervals. As already stated, such structures are called sheaves of vector spaces. For our work, we need to make this notion more precise: Let $\Intp$ denote the category (the set) of compact intervals $[a, b]$ with $a < b$, ordered by inclusion. Here, the abbreviation $\Intp$ stands for \emph{proper intervals}. The (pre)sheaves (of vector spaces) introduced in the \hyperref[sec:intro]{introduction} are called \emph{$\Intp$-(pre)sheaves (of vector spaces)} or alternatively \emph{(pre)sheaves (of vector spaces) on $\Intp$}. If the underlying category is clear from context, we may omit it, which justifies the use of \eanf{(pre)sheaves (of vector spaces)} in the \hyperref[sec:intro]{introduction}.

    $\Intp$-sheaves of vector spaces are insufficient for a rigorous study of controllability. As noted in the \hyperref[sec:intro]{introduction}, deciding whether a trajectory $f$ on $[a, b]$ and a trajectory $g$ on $[b, c]$ can be concatenated to form a solution on $[a, c]$ requires checking their \eanf{compatibility} at $b$. However, $[b, b]$ is not an object in the category $\Intp$. Consequently, the condition for compatibility cannot be formulated using only the algebraic structure of $\Intp$-sheaves of vector spaces. Instead, this compatibility check must be carried out \emph{externally}, by verifying continuity or differentiability of specific orders. In other words, the structure of $\Intp$-sheaves of vector spaces does not encapsulate the necessary structure that we expect from our solutions.

    The resolution is to transform this external analytic condition into an internal algebraic one. To do this, the regularity of the connection parametrized by $L$ must directly be encoded into the sheaf structure. We do this by extending our domain to the category $\Int$ of \emph{all} compact intervals $[a, b]$ with $a \leq b$, including degenerate intervals (points).

    \begin{Notation}
        Many constructions in this paper are formulated both for the categories $\Intp$ and $\Int$. Since the definitions and arguments are essentially identical in these two cases, we introduce the symbol $\Intpp$ as a placeholder. Whenever $\Intpp$ appears in a statement, it should be read as either $\Intp$ or $\Int$, and the statement is understood to hold for both choices.
    \end{Notation}

    An \emph{$\Int$-(pre)sheaf (of vector spaces)} is defined analogously to a $\Intp$-(pre)sheaf (of vector spaces), where every \eanf{$<$} in \hyperref[prop]{Properties}~\ref{prop:i}--\ref{prop:v} is replaced by \eanf{$\leq$}.

    A \emph{subpresheaf} $\mathcal{F}$ of a presheaf $\G$ on $\Intpp$ is a presheaf on $\Intpp$ such that
    \[
        \mathcal{F}[a, b] \subseteq \G[a, b] \qquad \text{for all $[a, b] \in \Intpp$}
    \]
    with restriction maps inherited from $\G$. We write $\mathcal{F} \preccurlyeq \G$.

    A \emph{subsheaf} $\mathcal{F}$ of a sheaf $\G$ on $\Intpp$ is a subpresheaf of $\G$ that is a sheaf. We write $\mathcal{F} \subsheaf \G$.

    In this work, the presheaves $\mathcal{F}$ considered are almost always presheaves of functions, i.e., for each object $U$ of the underlying category, $\mathcal{F}(U)$ is a set consisting of functions with domain $U$. In general, this need not be the case. For this reason, the elements of $\mathcal{F}(U)$ are generally called \emph{sections} (and not \eanf{functions}).

    \begin{Notation}\label{not:restriction}
        As before, for a $\Intpp$-presheaf $\mathcal F$ and an inclusion $[b, c]\subseteq[a, d]$ in $\Intpp$, we denote the corresponding restriction map by
        \[
            \cdot\rest{[b, c]}\colon \quad \mathcal F[a, d]\longrightarrow\mathcal F[b, c],\quad f\longmapsto f\rest{[b, c]}.
        \]
        The notation $\cdot\rest{[b, c]}$ generalizes ordinary function restriction but refers throughout to the presheaf restriction map. For the presheaves of functions considered below, both notions agree when $b < c$; for $b = c$, however, $\cdot\rest{[b, b]}$ generally does not denote function restriction to $[b, b]=\{b\}$.
    \end{Notation}

    The primary objective of this section is to extend the $\Intp$-sheaf $\Ca^{L, M, N}$ defined in the \hyperref[sec:intro]{introduction} to a sheaf on the category $\Int$. To maintain notational clarity, we rename the existing $\Intp$-sheaves as
    \[
        \boxed{\CIntp{L, M, N} \coloneqq \Ca^{L, M, N}} \qquad \text{and} \qquad \boxed{\CIntp{N} \coloneqq \Ca^N}
    \]
    while the new $\Int$-sheaves are denoted by
    \[
        \CInt{L, M, N} \qquad \text{and} \qquad \CInt{N},
    \]
    respectively.

    \begin{Remark}\label{rem:Intp_subsheaf}
        Using the subsheaf relation $\subsheaf$, we can compactly write
        \[
            \CIntp{L', M', N'} \subsheaf \CIntp{L, M, N} \quad \Longleftrightarrow \quad L' \geq L \ \land\ M' \geq M\ \land\ N' \geq N. \qedhere
        \]
    \end{Remark}

    For an $\Int$-presheaf $\mathcal{F}$, two sections $f \in \mathcal{F}[a, b]$ and $g \in \mathcal{F}[b, c]$ are compatible if and only if
    \[
        f\rest{[b, b]} = g\rest{[b, b]}.
    \]
    Recall from Notation~\ref{not:restriction} that these expressions denote restriction maps of the presheaf, not necessarily ordinary restrictions of functions. If $\mathcal{F}$ is a sheaf, $f$ and $g$ can then be glued together to form a section on $[a, c]$. The equality above is an internal algebraic condition. For this condition to correctly model the analytic requirement of $\Ca^L$-regularity, the restriction to the point $[b, b]$ must capture exactly the value and the first $L$ derivatives. Therefore, $\Ca^L$-compatibility can only be algebraically described for $\Int$-presheaves of at least $L$-times differentiable functions.

    Let $\mathcal{F}$ be a subsheaf of $\left(\CIntp{N}\right)^n$, where $L \leq N$. We extend $\mathcal{F}$ to $\Int$ by enforcing that $\Ca^L$-compatible sections are glueable.

    Let $f \in \mathcal{F}[a, b]$, $g \in \mathcal{F}[b, c]$, and $f$ and $g$ be $\Ca^L$-compatible in $b$, i.e.,
    \begin{equation}\label{eq:IV}
        f(b) = g(b), \quad \dot f(b) = \dot g(b), \quad \dots, \quad f^{(L)}(b) = g^{(L)}(b).
    \end{equation}
    In order to formulate the gluing at $t = b$ as a sheaf-theoretic gluing property, we must extend the definition of $\mathcal{F}$ to include degenerate intervals $[b, b]$ and define restriction maps to such intervals so that
    \[
        g\rest{[b, b]} = f\rest{[b, b]} \quad \Longleftrightarrow \quad \text{condition~\eqref{eq:IV} holds},
    \]
    i.e., such that it makes sense to call $f$ and $g$ compatible. To this end, we extend the sheaf~$\mathcal{F}$ to the category $\Int$ of (all) compact intervals, ordered by inclusion, by defining
    \[
        \ext_L\left( \mathcal{F} \right)[a, b] \coloneqq
        \begin{cases}
            \mathcal{F}[a, b], & a < b, \\
            \{b\} \times (\R^n)^{L + 1}, & a = b,
        \end{cases}
    \]
    with restriction maps inherited from $\mathcal{F}$ on nondegenerate intervals, and restriction to degenerate intervals defined via jet evaluation. Specifically, for $a < c$ and $b \in [a, c]$, we define
    \[\arraycolsep=0pt\begin{array}{rrcl}
        \cdot \rest{[b, b]} \colon {}&{} \ext_L \left( \mathcal{F}\right) [a, c] {}&{} \to {}&{} \ext_L \left( \mathcal{F} \right) [b, b] \\
        {}&{} f {}&{}\mapsto {}&{} f\rest{[b, b]} \coloneqq \left(b, f(b), \dot f(b), \dots, f^{(L)}(b)\right).
    \end{array}\]
    For $a = b = c$, $\cdot\rest{[b, b]}$ is just the identity. The idea behind including $b$ as the first entry in $(b, f(b), \dot f(b), \dots, f^{(L)}(b))$ is that it naturally represents a function $[b, b] \to \R^{n(L + 1)}$. The extension $\ext_L$ is well-defined if and only if $L \in \{-1, 0, \dots, N\}$, where $L = \infty$ means
    \[
        \ext_\infty\left( \mathcal{F} \right)[b, b] = \{b\} \times \left(\R^n\right)^{\N_0}.
    \]
    Clearly, $\ext_L \left( \mathcal{F} \right)$ defines a presheaf of vector spaces on the category $\Int$. Two functions
    \[
        f \in \ext_L \left( \mathcal{F} \right)[a, c] \quad \text{and} \quad g \in \ext_L \left( \mathcal{F} \right)[b, d],
    \]
    with $a \leq b \leq c \leq d$, are \emph{compatible} if and only if $f\rest{[b, c]} = g\rest{[b, c]}$. In general, not every compatible pair (the bad case is $b = c$) can be uniquely glued. This is the case, for example, if $\mathcal{F} = \left(\CIntp{N}\right)^n$ and $L < N$. Hence, $\ext_L \left( \mathcal{F} \right)$ is generally a presheaf but not a sheaf on $\Int$.

    Fortunately, there is a well-established and powerful tool called \emph{sheafification}, denoted by $(\cdot)^\#$. Intuitively, sheafification acts as a closure operator under gluing, thereby converting any presheaf into a sheaf. In particular, sheafification leaves sheaves unchanged.

    In our setting, $(\ext_L(\mathcal{F}))^\#[a, c]$ consists of all sections on $[a, c]$ that can be constructed through a finite number of gluing operations applied to compatible sections of $\ext_L(\mathcal{F})$. For example,
    \[
        f \cup g \in (\ext_L(\mathcal{F}))^\#[a, c].
    \]
    Note that \eanf{$f \cup g$} is a slight abuse of notation, as $f$ and $g$ are strictly speaking equivalence classes. When $L \geq 0$, this causes no confusion because we identify equivalence classes with their unique $\Ca^L$ representatives. However, for $L = -1$, the sections $f$ and $g$ do not behave as functions; for instance, point evaluation is not defined in $\left(\CIntp{-1}\right)^n[a, b]$. In this case, $f \cup g$ denotes the equivalence class formed by the union of two representatives of $f$ and $g$ that agree at $b$.

    Since $\ext_L\left(\CIntp{L ,M, N}\right)$ already admits unique gluing of compatible sections, it is an $\Int$-sheaf and is therefore unchanged by sheafification. We define
    \[
        \CInt{L, M, N} \coloneqq  \ext_L \left( \CIntp{L, M, N} \right) = \left( \ext_L \left( \CIntp{L, M, N} \right) \right)^\#
    \]
    and
    \[
        \CInt{N} \coloneqq \CInt{N, N, N}.
    \]
    More generally, we have, for all $L' = L, \dots, M$,
    \begin{equation}\label{eq:sheafification_general}
        \CInt{L, M, N} = \ext_L \left( \CIntp{L, M, N} \right) = \left( \ext_L \left( \CIntp{L', M, N} \right) \right)^\#.
    \end{equation}
    Here, for $L'>L$, sheafification adds $\Ca^L$-compatible gluings. As usual, we write $\left( \CInt{L, M, N}\right)^n$ for the $\Int$-sheaf with
    \[
        \left( \CInt{L, M, N}\right)^n[a, b] \coloneqq \left( \CInt{L, M, N}[a, b]\right)^n
    \]
    and component-wise restriction maps.

    The $\Int$-presheaf $\B_R$ consisting of $\left( \CInt{L, M, N} \right)^n$-solutions to $R(\ddt) w = 0$,
    \[
        \B_R \coloneqq  \ext_L \left( \mathcal{S}_R^{L, M, N} \right),
    \]
    is also a sheaf. In particular, $\B_R$ is a subsheaf of $\left( \CInt{L, M, N}\right)^n$,
    \[
        \B_R \subsheaf \left( \CInt{L, M, N}\right)^n.
    \]

    \begin{Remark}
        In contrast to Remark~\ref{rem:Intp_subsheaf}, the $\Int$-sheaves $\CInt{L', M', N'}$ and $\CInt{L, M, N}$ can only be compared when $L = L'$. We have
        \[
            \CInt{L, M', N'} \subsheaf \CInt{L, M, N} \quad \Longleftrightarrow \quad M' \geq M \ \land \ N' \geq N.
        \]
        Intuitively, this should not come as a surprise since $L$ determines the regularity required for compatibility. Hence, for $L' > L$, $\CInt{L', M', N'}$ and $\CInt{L, M, N}$ describe very different notions of compatibility for their sections, namely $\Ca^{L'}$- and $\Ca^L$-compatibility.
    \end{Remark}

    From a practical point of view, one is only interested in solutions on $\Intp$ but still wants to use the gluing property as described in $\Int$. This leads to the following definition.

    \begin{Definition}
        A subsheaf $\B \subsheaf \left(\CIntp{L}\right)^n$ is called \emph{$L$-regular} if and only if $\ext_L \left(\B\right)$ is a sheaf of vector spaces on $\Int$. Every such sheaf $\ext_L \left(\B\right)$ is called an \emph{$L$-regular local behaviour}. If $L$ is not specified, the sheaf is just called a \emph{local behaviour}.
    \end{Definition}

    \begin{Example}
        The subsheaf
        \[
            \mathcal{S}_R^{L, M, N} \subsheaf \left( \CIntp{L, M, N}\right)^n \subsheaf \left(\CIntp{L}\right)^n
        \]
        is $L$-regular since
        \[
            \B_R = \ext_L \left( \mathcal{S}_R^{L, M, N} \right)
        \]
        is a sheaf of vector spaces on $\Int$. In particular, $\B_R$ is an $L$-regular local behaviour.
    \end{Example}

    \begin{Remark}
        As we will see in Example~\ref{ex:counterexample_regular}, being $L$-regular is rather restrictive for subsheaves of $\left(\CIntp{L}\right)^n$.

        For subsheaves of $\left(\CInt{L}\right)^n$, however, the opposite is true: An $\Int$-sheaf $\B \subsheaf \left(\CInt{L}\right)^n$ is an $L$-regular local behaviour if and only if $\B = \ext_L \left( \B \rest{\Intp} \right)$, i.e., if and only if
        \[
            \forall t \in \R \colon \quad \B[t, t] =
            \begin{cases}
                \{t\} \times (\R^n)^{L+1}, & L < \infty, \\
                \{t\} \times (\R^n)^{\N_0}, & L = \infty.
            \end{cases}
        \]
        In other words, any subsheaf $\B \subsheaf \left(\CInt{L}\right)^n$ induces the (unique) $L$-regular local behaviour
        \[
            \ext_L \left( \B \rest{\Intp} \right).
        \]
        Since solutions at single points are not of primary interest in this work, we can always assume, without loss of generality, that our subsheaves of $\left(\CInt{L}\right)^n$ are local behaviours.
    \end{Remark}

    \begin{Example}\label{ex:counterexample_regular}
        The map $\ext_L$ is generally not a suitable choice for producing an $\Int$-sheaf from a $\Intp$-sheaf. In fact, some $\Intp$-sheaves admit no extension to an $\Int$-sheaf.
        \begin{enumerate}[label=(\roman*{})]
            \item\label{ex:counterexample_regular_(i)} For finite $L$, not every subsheaf $\B \subsheaf \left( \CIntp{L} \right)^n$ is $L$-regular. A simple counterexample is $\B = \left( \CIntp{L+1}\right)^n$: By equation~\eqref{eq:sheafification_general}, the sheafification of $\ext_L \B$ is $\left(\CInt{L,L+1,L+1}\right)^n$. Concretely, for $n=1$, the sections $0\in\B[-1, 0]$ and $t^{L+1}\in\B[0, 1]$ are compatible in $\ext_L \B$, but cannot be glued to a section of $\B[-1, 1]$, since their concatenation is not in $\Ca^{L+1}[-1,1]$.

            \item\label{ex:counterexample_regular_(ii)} Even subsheaves $\B \subsheaf \left( \CIntp{\infty} \right)^n$ may fail to be $\infty$-regular. Consider the function $f \in \Cinf(\R)$ defined by
            \[
                f(t) \coloneqq \begin{cases}
                    \mathrm{e}^{-t^{-2}}, & t \neq 0, \\
                    0, & t = 0,
                \end{cases}
            \]
            and define the subsheaf $\B$ by $\B[a, b] \coloneqq \big\{ c \cdot f\rest{[a, b]} \mid c \in \R \big\}$. Then $\ext_{\infty}(\B)$ is a presheaf but not a sheaf on $\Int$: the sections $0 \in \B[-1, 0]$ and $f\rest{[0, 1]} \in \B[0, 1]$ restrict to the zero-jet in $\ext_\infty(\B)[0,0]$ and hence are compatible, but they cannot be glued. The issue is that~$c$ cannot be recovered from derivative data at $0$ alone. Compare to Example~\ref{ex:presheaf_not_sheaf}, where restriction to the overlap also loses the scalar coefficient $c$.

            Nevertheless, $\B$ can be extended to an $\Int$-sheaf by setting $\B[t, t]\coloneqq\R$ and defining
            \[
                c f\rest{[a, b]} \mapsto \left( c f\rest{[a, b]} \right)\rest{[t, t]} \coloneqq c.
            \]

            \item\label{ex:counterexample_regular_(iii)} The following example shows that a $\Intp$-sheaf $\B$ need not admit any extension to an $\Int$-sheaf. Define $\B$ by
            \[
                \B[a,b]\coloneqq
                \begin{cases}
                    \R, & a<0<b,\\
                    0, & \text{otherwise},
                \end{cases}
            \]
            with identity restrictions between nonzero section spaces and zero restrictions otherwise. For $a < 0 < c$, the map $\B[a, c] \to \B[a, 0] \oplus \B[0, c]$ is the zero map $\R \to 0$ and hence not injective. Thus, in any extension, the zero sections on $[a, 0]$ and $[0, c]$ would not be uniquely glueable. \qedhere
        \end{enumerate}
    \end{Example}

    Example~\ref{ex:counterexample_regular}~\ref{ex:counterexample_regular_(i)} and~\ref{ex:counterexample_regular_(ii)} show that $\ext_L$ need not convert a subsheaf of $\left(\CIntp{L}\right)^n$ to an $\Int$-sheaf, whereas~\ref{ex:counterexample_regular_(iii)} shows that an $\Int$-extension need not exist at all. The next theorem characterizes those $\Intp$-sheaves that admit an extension to an $\Int$-sheaf. The construction used in its proof generalizes the successful construction in~\ref{ex:counterexample_regular_(ii)}.

    \begin{Theorem}\label{thm:Int&Intp}
        Let $\B$ be an $\Int$-sheaf of vector spaces. Then $\B\rest{\Intp}$ is a $\Intp$-sheaf and
        \begin{equation}\label{eq:sheaf_ext}
            \forall a < b < c \colon \quad \B[a, c] \to \B[a, b] \oplus \B[b, c],\ f \mapsto \big(f\rest{[a, b]}, f\rest{[b, c]} \big) \quad \text{is injective.}
        \end{equation}
        Conversely, every $\Intp$-sheaf satisfying property~\eqref{eq:sheaf_ext} can be extended to an $\Int$-sheaf.
    \end{Theorem}
    \begin{proof}
        The proof uses categorical limit constructions and is therefore deferred to the \hyperref[app]{appendix}.
    \end{proof}

    In terms of solutions to differential equations, however, the abstract extensions provided by Theorem~\ref{thm:Int&Intp} are typically of limited interest. While for real-world physical systems it is debatable whether solutions that coincide to all orders at a point should be considered glueable, for time-varying linear systems of the form $R(\ddt)w = 0$, this is always the case. Hence, we believe it is sufficient to restrict our attention to $L$-regular $\Intp$-sheaves.

\section{Controllability}\label{sec:contr}
    \noindent This section applies the sheaf-theoretic framework to systems theory, with a focus on controllability. The most general controllability results are stated in Subsection~\ref{subsec:contr_main}.

    In Subsection~\ref{subsec:sheaf_morphisms}, we explain $R(\ddt)$ as a sheaf morphism. This prepares us for Subsection~\ref{subsec:tnnf}, where we formulate an algebraic criterion for controllability (Theorem~\ref{thm:main}), translating the results of \citet[Theorem 3.2]{part1} and \citet[Theorems~6,~7]{zerz} into the sheaf-theoretic setting.

    To enable such a characterization, the analysis is restricted to systems with real meromorphic coefficients, which permits the use of the Teichmüller--Nakayama form, a factorization for differential polynomial matrices. Furthermore, the class of admissible solutions is restricted to ensure a continuation property.

    \begin{Definition}(Continuation property)\label{def:continuation}
        Let $\mathcal{S}$ be an $\Int$-presheaf of vector spaces. Then, \emph{$\mathcal{S}$ satisfies the continuation property} if and only if
        \begin{equation*}
            \forall [a, b] \in \Intp \; \forall g \in \mathcal{S}[a, b] \; \exists \varepsilon > 0 \; \exists f \in \mathcal{S}[a - \varepsilon, a]  \; \exists h \in \mathcal{S}[b, b + \varepsilon] \colon \quad f, g, h \text{ compatible} \qedhere
        \end{equation*}
    \end{Definition}

    Intuitively, the \hyperref[def:continuation]{continuation property} states that any solution defined on $[a,b]$ can be extended to the left and right; in other words, the maximal existence interval of a solution is open.

    In Subsection~\ref{subsec:global_comp}, we discussed two works related to continuation questions: \citet{Schultz2020} introduce a similar extension property. \citet{Wasow1987} develops an extensive continuation and asymptotic theory for analytic differential equations near singular points. The scope of the latter theory illustrates the difficulty of continuation questions even under more restrictive analytic assumptions. Developing a comprehensive continuation theory is beyond the scope of the present work. We therefore restrict ourselves to sufficient conditions relevant to controllability: Theorem~\ref{thm:contr}~\ref{thm:contr:(i)} shows that controllability implies the continuation property, and Lemma~\ref{lemma:cont_degree} provides an algebraically verifiable sufficient condition.

    \subsection{Controllability of sheaves}\label{subsec:contr_main}

    \begin{Definition}[Controllability]\label{def:contr_new}
        Let $\mathcal{S}$ be an $\Int$-presheaf of vector spaces. We define the different notions of controllability given in Table~\ref{tab:contr_new}.

        \begin{table}[!htbp]
            \centering
            \setcellgapes{4pt}
            \makegapedcells

            \begin{tabular}{|c|c|c|}
                \hline                Definiendum & Definiens\\
                \hline \hline                \makecell[c]{$f \in \mathcal{S}[a, b]$ is controllable \\
                to $h \in \mathcal{S}[c, d]$ via $g \in \mathcal{S}[b, c]$}
                & $f$, $g$, and $h$ are compatible\\
                \hline                \makecell[c]{$f \in \mathcal{S}[a, b]$ is controllable\\
                to $h \in \mathcal{S}[c, d]$}
                & \makecell[c]{$\exists g \in \mathcal{S}[b, c]\colon f$ is controllable to $h$ via $g$}\\
                \hline                \makecell[c]{$\mathcal{S}$ is controllable to $0$}
                & \makecell[c]{$\forall a < b < c < d \; \forall f \in \mathcal{S}[a, b]\colon$ \\
                $f$ is controllable to $0 \in \mathcal{S}[c, d]$}\\
                \hline                \makecell[c]{$\mathcal{S}$ is controllable}
                & \makecell[c]{$\forall a < b < c < d \; \forall f \in \mathcal{S}[a, b]$ \\
                $\forall h \in \mathcal{S}[c, d]\colon f$ is controllable to $h$}\\
                \hline            \end{tabular}
            \vspace{1em}
            \caption{Notions of controllability for presheaves} \vspace{-1em-14pt} \qedhere
            \label{tab:contr_new}
        \end{table}
    \end{Definition}

    \begin{Remark}
        In the global definition of controllability in \cite[Definition~3]{zerz}, recalled as property~\eqref{eq:contr_zerz} in Subsection~\ref{subsec:global_comp}, one requires only the existence of a transition time $\tau>0$. At first sight, this may appear weaker than the quantification $\forall a<b<c<d$ in Definition~\ref{def:contr_new}. The essential restriction in property~\eqref{eq:contr_zerz}, however, is that all three solutions $w_1$, $w_2$, and $w$ must be smooth on the transition interval $[t_0,t_0+\tau]$. By deliberately introducing a singularity into $w_1$ or $w_2$ at a point $c\in(t_0,t_0+\tau)$, one can force the transition to be completed before $c$. Consequently, once a common smoothness interval $\I\ni t_0$ of $w_1$ and $w_2$ is fixed, the existential quantifier \eanf{$\exists\tau>0$ with $t_0+\tau\in\I$} may be replaced by \eanf{$\forall\tau>0$ with $t_0+\tau\in\I$}. We do not spell out this direct argument in further detail: the equivalence of the two notions also follows indirectly from Corollary~\ref{cor:main} and \cite[Theorems~6,~7]{zerz}.
    \end{Remark}

    It is straightforward to see that the \hyperref[def:continuation]{continuation property} and controllability only have to be checked for respective presheaves. We formulate this in the following proposition.

    \begin{Proposition}\label{prop:presheaves_ftw}
        Let $\B \subsheaf \left( \CIntp{N} \right)^n$ and $\B^{L} \coloneqq \ext_{L} \left(\B \right)$ for $-1 \leq L \leq N$. If $\B^N$ satisfies any of the following properties: the \hyperref[def:continuation]{continuation property}, controllability, or controllability to $0$, then $\B^L$ and $(\B^L)^\#$ satisfy the same property for all $-1 \leq L \leq N$.
    \end{Proposition}

    \begin{proof}
         Fix $-1 \leq L \leq N$. For $a < b < c < d$, if $f \in \B^L[a, b]$, $g \in \B^L[b, c]$, and $h \in \B^L[c, d]$ are compatible as sections in $\B^N$, i.e.,
        \[
            \left(b, f(b), \dots f^{(N)}(b) \right) = \left(b, g(b), \dots g^{(N)}(b)\right),\ \left(c, g(c), \dots g^{(N)}(c) \right) = \left(c, h(c), \dots h^{(N)}(c)\right),
        \]
        compatibility follows in $\B^L$:
        \[
            \left(b, f(b), \dots f^{(L)}(b) \right) = \left(b, g(b), \dots g^{(L)}(b)\right),\ \left(c, g(c), \dots g^{(L)}(c) \right) = \left(c, h(c), \dots h^{(L)}(c)\right).
        \]
        Hence, all results follow for $\B^L$.

        Now, consider $(\B^L)^\#$. Every $f \in (\B^L)^\#[a, b]$ has an initial piece
        \[
            f_0 \coloneqq f\rest{[a, a + \varepsilon]}\in \B^L[a, a+ \varepsilon] = \B^N[a, a+ \varepsilon]
        \]
        and an end piece
        \[
            f_1 \coloneqq f\rest{[b - \varepsilon, b]}\in \B^L[b - \varepsilon, b] = \B^N[b - \varepsilon, b],
        \]
        where $\varepsilon > 0$ is sufficiently small.

        For the \hyperref[def:continuation]{continuation property}, apply it to $f_0$ and $f_1$ to obtain suitable continuations of~$f$.

        For controllability (resp.\ controllability to $0$), control the end piece of $f\in(\B^L)^\#[a, b]$ to the initial piece of $h\in(\B^L)^\#[c, d]$ (resp.\ $h = 0\in(\B^L)^\#[c, d]$) via $g \in \B^N[b, c]$. The compatibility is of regularity $\Ca^N$ in $b$ and $c$, hence $f$, $g$, and $h$ are (sufficiently) compatible.
    \end{proof}

    \begin{Remark}
        While we have defined $\ext_L$ and the sheafification of $\ext_L(\B)$ only for subsheaves $\B \subsheaf \left( \CIntp{N} \right)^n$ with $N \geq L$, these constructions naturally generalize to arbitrary subpresheaves $\B$. Indeed, Proposition~\ref{prop:presheaves_ftw} remains valid for any subpresheaf $\B \preccurlyeq \left( \CIntp{N} \right)^n$.
    \end{Remark}

    Local behaviours that satisfy the \hyperref[def:continuation]{continuation property} admit the usual relation between controllability and controllability to $0$, as shown in the following theorem.

    \begin{Theorem}\label{thm:contr}
        Let $\B$ be a local behaviour. Then the following statements hold:
        \begin{enumerate}[label=(\roman*)]
            \item \label{thm:contr:(i)} If $\B$ is controllable, then it satisfies the \hyperref[def:continuation]{continuation property}.
            \item \label{thm:contr:(ii)} If $\B$ satisfies the \hyperref[def:continuation]{continuation property}, then $\B$ is controllable if and only if it is controllable to $0$.
            \item \label{thm:contr:(iii)} $\B$ is controllable if and only if it is controllable to $0$ and satisfies the \hyperref[def:continuation]{continuation property}.
        \end{enumerate}
    \end{Theorem}

    \begin{proof}
        \ref{thm:contr:(i)}: Let $f \in \B[a, b]$, where $a < b$. Choose any $\varepsilon > 0$. Since $\B$ is controllable and a presheaf of vector spaces, $0 \in \B[a - 2\varepsilon, a - \varepsilon]$ is controllable to $f$ via some $g_{-}$ and $f$ is controllable to $0 \in \B[b + \varepsilon, b + 2\varepsilon]$ via some $g_{+}$. Now, $g_{-}$, $f$, and $g_{+}$ are compatible, which proves the \hyperref[def:continuation]{continuation property}.

        \ref{thm:contr:(ii)}: Assume that $\B$ satisfies the \hyperref[def:continuation]{continuation property}. Controllability to $0$ is necessary for controllability. It remains to prove its sufficiency.

        Consider $a < b < c < d$ and functions $f \in \B[a, b]$ and $h \in \B[c, d]$. By assumption, $f$ is controllable to $0 \in \B[c, d]$. That means, there exists $g_1 \in \B[b, c]$ that is compatible with both $f$ and $0$, see Figure~\ref{fig:proof}. The glued function~$e_1$ is defined by $e_1 \coloneqq f \cup g_1 \cup 0 \in \B[a, d]$.

        By the \hyperref[def:continuation]{continuation property}, $h$ can be extended to the left, yielding $\hat{h} \in \B[c - 2\varepsilon, d]$ for some $\varepsilon > 0$. Select $\varepsilon$ small enough to ensure $b < c - 2\varepsilon$. Now, consider $e_1\rest{[c - 2\varepsilon, d]} - \hat{h} \in \B[c - 2\varepsilon, d]$. The initial segment $f_2$ (depicted in blue in Figure~\ref{fig:proof}) of this difference is defined as
        \[
        f_2 \coloneqq e_1\rest{[c - 2\varepsilon, c - \varepsilon]} - \hat{h}\rest{[c - 2\varepsilon, c - \varepsilon]} \in \B[c - 2\varepsilon, c - \varepsilon].
        \]
        By assumption, $f_2$ is controllable to $0 \in \B[c, d]$, yielding a compatible function $g_2 \in \B[c - \varepsilon, c]$ and a glued function $e_2 \coloneqq f_2 \cup g_2 \cup 0 \in \B[c - 2\varepsilon, d]$. Then,
        \[
        h_3 \coloneqq \hat{h} + e_2 \in \B[c - 2\varepsilon, d]
        \]
        coincides with $g_1$ on $[c - 2\varepsilon, c - \varepsilon]$ and with $h$ on $[c, d]$. Thus, $f$, $g_1\rest{[b, c - 2\varepsilon]}$, and $h_3$ are compatible and can be glued together to form $e_3 \coloneqq f \cup g_1 \rest{[b, c - 2\varepsilon]} \cup h_3 \in \B[a, d]$. Setting $g \coloneqq e_3\rest{[b, c]}$, we conclude that $f$, $g$, and $h$ are compatible because $e_3 = f \cup g \cup h$.

        \begin{figure}[H]
            \centering
            \scalebox{1.1}{\includegraphics{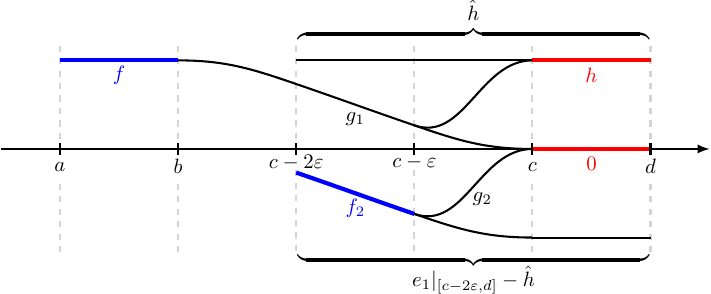}}
            \caption{Sketch for the proof of Theorem~\ref{thm:contr}~\ref{thm:contr:(ii)}.}
            \label{fig:proof}
        \end{figure}

        \ref{thm:contr:(iii)}: This follows directly from \ref{thm:contr:(i)} and \ref{thm:contr:(ii)}.
    \end{proof}

    We conclude with several remarks concerning the preceding theorem.

    \begin{Remark}
        \begin{enumerate}[label=(\roman*)]
            \item The statement of Theorem~\ref{thm:contr}~\ref{thm:contr:(i)} holds more generally for $\Int$-presheaves of vector spaces.

            \item In the proof of Theorem~\ref{thm:contr}~\ref{thm:contr:(i)}, the choice of $\varepsilon>0$ was arbitrary. This yields the stronger conclusion that, for every controllable local behaviour $\B$, the restriction $\B\rest{\Intp}$ is a \emph{flabby} $\Intp$-sheaf, meaning that all restriction maps between proper intervals are surjective. The continuation property is weaker than flabbiness, since it only requires each section to admit an extension to some larger interval.

            \item The converse of Theorem~\ref{thm:contr}~\ref{thm:contr:(i)} is clearly false, as the $\Int$-sheaf $\B$ from Example~\ref{ex:counterexample_regular}~\ref{ex:counterexample_regular_(ii)} satisfies the \hyperref[def:continuation]{continuation property} but is not controllable.

            \item In Theorem~\ref{thm:contr}~\ref{thm:contr:(ii)} and~\ref{thm:contr:(iii)}, the \hyperref[def:continuation]{continuation property} cannot be omitted, as the following example shows.

            Let $\B$ be a nonzero controllable $L$-regular local behaviour with $L \geq 0$, e.g., $\B = \CInt{\infty}$, and let $s \in \B[a, b]$ be a nonzero section. Thus, there exists a point $t \in (a, b)$ such that $s(t) \neq 0$. We define a new local behaviour $\B'$ that is derived from the sheaf $\B$ by setting it to $0$ left of $t$. Formally, we define
            \[
                \B'[c, d] \coloneqq \begin{cases}
                    \B[c, d], & t \leq c \, \lor \, c = d, \\
                    \bigset{f \in \B[c, d]}{f\rest{[c, t]} = 0}, & c < t \leq d, \\
                    \{0\}, & c < d < t,
                \end{cases} \qquad \text{for } [c, d] \in \Int.
            \]
            This is an $\Int$-sheaf as gluing at $t$ remains possible and the restriction morphisms are well-defined. Moreover, $\B'$ is controllable to $0$ since any $f \in \B'[c, d]$ with $c < d$ is either already $0$ or satisfies $f\rest{[d - \varepsilon, d]} \in \B[d - \varepsilon, d]$, where $t \leq d - \varepsilon$. In the latter case, we rely on the controllability of $\B$ to demonstrate that $f$ is controllable to $0$. However, $\B'$ is not controllable, as $0 \in \B'[a-1, a]$ cannot be controlled to $s\rest{[t, b]} \in \B'[t, b]$. \qedhere
        \end{enumerate}
    \end{Remark}

    \subsection{Differential operators as sheaf morphisms}\label{subsec:sheaf_morphisms}

    \noindent In the previous subsection, we derived controllability results for general local behaviours. To obtain specific results for time-varying linear systems $R(\ddt) w = 0$, we interpret the differential operator $R(\ddt)$ as a morphism of sheaves. Identifying the solution space as the kernel of this morphism prepares for the algebraic analysis in Subsection~\ref{subsec:tnnf}.

    \begin{Definition}
        A \emph{morphism of $\Intpp$-presheaves (of vector spaces)}
        \[
            \varphi = (\varphi_{[a, b]})_{[a, b] \in \Intpp} \colon \mathcal{F} \to \mathcal{G},
        \]
        where $\mathcal{F}$ and $\mathcal{G}$ are presheaves (of vector spaces) on $\Intpp$, is a family of (linear) maps
        \[
            \varphi_{[a, b]} \colon \mathcal{F}[a, b] \to \mathcal{G}[a, b], \qquad [a, b] \in \Intpp,
        \]
        such that for all $[b, c] \subseteq [a, d]$ with $[b, c], [a, d] \in \Intpp$,
        \[\begin{tikzcd}
            \mathcal{F}[a, d] \arrow[r, "\varphi_{[a, d]}"] \arrow[d, "\cdot \rest{[b, c]}"'] & \mathcal{G}[a, d] \arrow[d, "\cdot \rest{[b, c]}"]\\
            \mathcal{F}[b, c] \arrow[r, "\varphi_{[b, c]}"'] & \mathcal{G}[b, c]
        \end{tikzcd}\]
        commutes.

        A \emph{morphism of $\Intpp$-sheaves (of vector spaces)} $\varphi \colon \mathcal{F} \to \mathcal{G}$ is a morphism of presheaves (of vector spaces), where $\mathcal{F}$ and $\mathcal{G}$ are sheaves on $\Intpp$.
    \end{Definition}

    \begin{Definition}
        Let $\mathfrak{R} \colon \mathcal{F} \to \mathcal{G}$ be a morphism of presheaves of vector spaces on $\Intpp$. Then, the collection of all sections being mapped to zero via $\mathfrak{R}$ constitutes a subpresheaf of $\mathcal{F}$ denoted by $\ker \mathfrak{R}$ and called the \emph{kernel of $\mathfrak{R}$}.
    \end{Definition}

    The following proposition shows that the kernel preserves the sheaf property.
    \begin{Proposition}
        Let $\mathfrak{R} \colon \mathcal{F} \to \mathcal{G}$ be a morphism of sheaves of vector spaces on $\Intpp$. Then, $\ker \mathfrak{R}$ is a sheaf of vector spaces. In particular, $\ker \mathfrak{R} \subsheaf \mathcal{F}$.
    \end{Proposition}
    \begin{proof}
        Let $a \leq b < c \leq d$ in the case of $\Intp$-sheaves and $a \leq b \leq c \leq d$ in the case of $\Int$-sheaves. Let $f \in \ker \mathfrak{R}[a, c] \subseteq \mathcal{F}[a, c]$ and $g \in \ker \mathfrak{R}[b, d] \subseteq \mathcal{F}[b, d]$ be compatible, i.e.,
        \[
            f\rest{[b, c]} = g\rest{[b, c]}.
        \]
        Since $\mathcal{F}$ is a sheaf, there exists a unique glued section $e \in \mathcal{F}[a, d]$ that restricts to $f$ on $[a, c]$ and to $g$ on $[b, d]$. It remains to show $\mathfrak{R}e = 0 \in \G[a, d]$. We have
        \[
            (\mathfrak{R}e) \rest{[a, c]} = \mathfrak{R}(e \rest{[a, c]}) = \mathfrak{R}f = 0\in\G[a, c]
        \]
        and analogously $(\mathfrak{R}e) \rest{[b, d]} = 0 \in \G[b, d]$. Since $(\mathfrak{R}e) \rest{[a, c]} = 0\in\G[a, c]$ and $(\mathfrak{R}e) \rest{[b, d]} = 0\in\G[b, d]$ are compatible and $\G$ is a sheaf, they uniquely glue to
        \[
            \mathfrak{R}e = 0. \qedhere
        \]
    \end{proof}

    \begin{Example}
        \begin{enumerate}[label=(\roman*)]
            \item The derivative $\ddt$ is a sheaf morphism
            \[
                \ddt \colon \CIntp{L, M, N} \to \CIntp{\max\{L - 1, -1\}, \max\{M - 1, -1\}, N - 1}, \qquad N \geq 1.
            \]
            It is also a sheaf morphism
            \[
                \ddt \colon \CInt{L, M, N} \to \CInt{\max\{L - 1, -1\}, \max\{M - 1, -1\}, N - 1}, \qquad N \geq 1,
            \]
            when setting it on degenerate intervals $[a, a]$ to
            \[
                \ddt \colon \{a\} \times \R^{L + 1} \to \{a\} \times \R^{\max\{L,\, 0\}},\ (a, x_0, \dots, x_L) \mapsto (a, x_1, \dots, x_{L}).
            \]
            In particular, for $L = -1$, $\ddt$ is the identity on degenerate intervals $[a, a]$ as it just maps $(a)$ to $(a)$. Note that $(a)$, or more precisely, $\{(a)\}$, is the constant zero function $[a, a] \to \R^0$.

            \item Let $R(\ddt)$ be a differential polynomial of degree $\rho$ with coefficients of sufficient regularity, e.g., $R_i \colon \R \to \R^{m \times n}$ so that
            \[
                R_i\rest{[a, b]} \in \left( \CIntp{L, M, N} \right)^{m \times n}[a, b] \qquad \text{for all $[a, b] \in \Intp$};
            \]
            by abuse of notation, we simply write $R_i \in \left( \CIntp{L, M, N} \right)^{m \times n}(\R)$. Then, for all $-1 \leq L' \leq M' \leq N' \in \{\rho, \rho+1, \ldots, \infty\}$, $R(\ddt)$ defines a sheaf morphism
            \[
                R(\ddt) \colon \left( \CIntp{L', M', N'} \right)^n \to \left( \CIntp{L'', M'', N''} \right)^m,
            \]
            where
            \[\arraycolsep=0pt\begin{array}{rcl}
                L'' {}&{} \coloneqq {}&{} \min\{L, \max\{L'-\rho, -1\}\}, \\
                M'' {}&{} \coloneqq {}&{} \min\{M, \max\{M'-\rho, -1\}\}, \\
                 N'' {}&{} \coloneqq {}&{} \min\{N, N'-\rho\}.
            \end{array}\]
            To extend this sheaf morphism to $\Int$-sheaves, one could define addition and multiplication for jets. Instead, we make use of the fact that
            \[
                \CIntp{L'', M'', N''} \subsheaf \CIntp{-1, -1, 0},
            \]
            and consider $R(\ddt)$ as a morphism
            \[
                R(\ddt) \colon \left(  \CIntp{L', M', N'} \right)^n \to \left( \CIntp{-1, -1, 0} \right)^m.
            \]
            Now, we may weaken our requirement of regularity to $R_i \in \left( \CIntp{-1, -1, 0} \right)^{m \times n}(\R)$. In particular, $R(\ddt)$ has a unique extension to a morphism of $\Int$-sheaves
            \[
                 R(\ddt)_{\Int} \colon \left(  \CInt{L', M', N'} \right)^n \to \left( \CInt{-1, -1, 0} \right)^m
            \]
            as every $(a, x_0, \dots, x_{L'}) \in \left(  \CInt{L', M', N'} \right)^n[a, a] = \{a\} \times \left(\R^n\right)^{L'+1}$ can only be mapped to $0 = (a) \in \left( \CInt{-1, -1, 0} \right)^m[a, a] = \{a\} \times \left(\R^m\right)^0$. In particular,
            \[
                \ker R(\ddt)_\Int  = \ext_{L'} \left( \ker R(\ddt)\right).
            \]
            Since $R(\ddt)$ can be defined on many different domains such as $\left(\CIntp{N}\right)^n$, $\left(\CIntp{L, M, N}\right)^n$ or $\left(\CInt{L, M, N}\right)^n$, we simplify our notation and write
            \[
                \ker^{N}_\Intp R(\ddt), \quad \ker^{L, M, N}_\Intp R(\ddt) \quad \text{or} \quad \ker^{L, M, N}_\Int R(\ddt)
            \]
            for the respective kernels. Using this notation, $\B_R = \ext_L \left( \ker^{L, M, N}_\Intp R(\ddt) \right)$ from Section~\ref{sec:loc_sol} satisfies
            \[
                \B_R = \ker^{L, M, N}_\Int R(\ddt). \qedhere
            \]
        \end{enumerate}

    \end{Example}

    \begin{Remark}
        For Subsection~\ref{subsec:tnnf}, stronger regularity assumptions are required. We elaborate on the precise reasons in that subsection, but in short, we must ensure that the derivative defines an endomorphism
        \[
            \ddt \colon \mathcal{F} \to \mathcal{F}.
        \]
        Among the regularity-triple sheaves $\mathcal C_{\mathbf{pInt}}^{L,M,N}$, differentiation is an endomorphism only for
        \[
            \CIntp{-1, -1, \infty}, \quad \CIntp{-1, \infty, \infty}, \quad \text{and} \quad \CIntp{\infty, \infty, \infty}.
        \]
        Of these, $\CIntp{-1, -1, \infty}$ is the largest and the only one containing all real meromorphic functions. Hence, our differential operators will always be defined on $\left(\CIntp{-1, -1, \infty}\right)^n$.

        Nevertheless, we consider solutions in $\ker^{L, \infty, \infty}_\Intp R(\ddt)$, since we want to specify the regularity of our solutions via $L$, and the condition $M = \infty$ ensures sufficient boundary regularity. The latter is used in Lemma~\ref{lemma:cont_degree} to ensure the \hyperref[def:continuation]{continuation property}. As we demonstrated in Theorem~\ref{thm:contr}~\ref{thm:contr:(i)}, the \hyperref[def:continuation]{continuation property} is a necessary condition for controllability and, consequently, for all controllability results in Subsection~\ref{subsec:tnnf}.
    \end{Remark}

    \subsection{Controllability of time-varying systems}\label{subsec:tnnf}

    \noindent
    In this subsection, we investigate controllability of dynamical systems with real meromorphic coefficients. A dynamical system with real meromorphic coefficients is of the form
    \[
        R(\ddt) w = \sum_{i = 0}^\rho R_i \ddtn{i} w = 0,
    \]
    where $R_i \in \M^{m \times n}$. By regarding real meromorphic functions as a subsheaf of $\CIntp{-1, -1, \infty}$, $R(\ddt)$ can be written as a sheaf morphism
    \[
        R(\ddt) \colon \left(\CIntp{-1, -1, \infty}\right)^n \to \left(\CIntp{-1, -1, \infty}\right)^m.
    \]
    Solutions are considered in $\ker^{L, \infty, \infty}_\Intp R(\ddt)$ or $\ker^{L, \infty, \infty}_\Int R(\ddt)$.

    \begin{Remark} \label{rem:why_inf}
        It is generally desirable to define the domain of $R(\ddt)$ as large as possible in order not to exclude any potential solutions. Ideally, one would thus consider $R(\ddt)$ as a sheaf morphism
        \[
            R(\ddt) \colon \left(\CIntp{-1, -1, N}\right)^n \to \left(\CIntp{-1, -1, 0}\right)^m,
        \]
        where $N = \rho$ is the degree of the differential operator $R(\ddt)$.

        However, to enable the application of algebraic methods, it is necessary to consider factorizations of the form $R(\ddt) = A(\ddt) B(\ddt)$. For instance, one would expect the identity
        \[
            \begin{bmatrix}
                1 &  & \\
                & 1 & \\
                & & 1
            \end{bmatrix}
            =
            \begin{bmatrix}
                1 &  & \\
                & -\ddtn{j} & 1 \\
                & 1 & 0
            \end{bmatrix}
            \begin{bmatrix}
                1 &  & \\
                & 0 & 1 \\
                & 1 & \ddtn{j}
            \end{bmatrix}, \qquad \text{for all } j \in \N,
        \]
        to hold under standard matrix multiplication rules. We explain later how such differential operators $R(\ddt)$ can be represented as matrices over an appropriate noncommutative ring.

        The example above illustrates that the inequality
        \[
            \deg R(D) \leq \deg A(D) + \deg B(D)
        \]
        holds in general. However, the sum on the right-hand side can be arbitrarily large. Consequently, it is in general not possible to apply $B(\ddt)$ and then $A(\ddt)$ to a section in $\CIntp{-1, -1, N}$. To avoid this problem, we set $N = \infty$.
    \end{Remark}

    We now recall a key algebraic tool for the analysis of dynamical systems with real meromorphic coefficients. First, we define our operators as matrices over a skew polynomial ring $\M[D]$. The ring $\M[D]$ consists of polynomials in the indeterminate $D$ with coefficients in $\M$, with usual addition and multiplication defined by
    \[
        D \cdot f = f \cdot D + \ddt f, \qquad \text{for } f \in \M.
    \]
    Under this identification, the structure of matrices $R(D) \in \M[D]^{m \times n}$ corresponds to that of operators
    \[
        R(\ddt) \colon \left(\CIntp{-1, -1, \infty}\right)^n \to \left(\CIntp{-1, -1, \infty}\right)^m,
    \]
    via substituting $\ddt$ for $D$. This formally defines the multiplication rule that was used in Remark~\ref{rem:why_inf}.

    The ring $\M[D]$ is both a left and right Euclidean domain, with the degree given by the usual polynomial degree (\citet[Chapter~0.8]{free_rings}). In particular, $\M[D]$ is a left and right principal ideal domain, meaning that every left or right ideal is generated by a single element.

    For $R(D)\in\MDmn$, this property enables a specific type of factorization known in the literature as the Jacobson form (\citet{zerz}) or the Teichmüller--Nakayama form (\citet{part1}). The factorization theorem has undergone several generalizations and refinements, as explained in \citet[Chapter 8]{free_rings}, which accounts for the varied nomenclature.

    Before stating the factorization theorem, we clarify the necessary notions. A polynomial matrix $A(D)\in\MDmn$ is called \emph{left invertible} (resp.\ \emph{right invertible}) if and only if there exists~a polynomial matrix $B(D)\in\MD{}{n\times m}$ such that $B(D)A(D) = I_n$ (resp.\ ${A(D)B(D) = I_m}$). A polynomial matrix $A(D)\in\MD{}{n\times n}$ is called \emph{unimodular} if and only if there exists a polynomial matrix $B(D)\in\MD{}{n\times n}$ such that $A(D)B(D) = B(D)A(D) = I_n$.

    \begin{Theorem}[Teichmüller--Nakayama form, {\cite[Chapter~8]{free_rings}}]
        \label{thm:tnnf}
		For any ${R(D)\in\MDmn}$ with $\rk\!_\MD{}{} R(D) = \ell \geq 1$, there exist unimodular matrices $U(D), V(D)$ over $\MD{}{}$ with sizes $m$ and $n$, respectively, such that
		\begin{equation}\label{eq:tnnf}
			R(D) = U(D)^{-1} \begin{bmatrix}I_{\ell-1}&&\\&r(D)&\\&&0_{(m-\ell)\times(n-\ell)}\end{bmatrix} V(D)^{-1},
		\end{equation}
		where $r(D)\in\MD{}{}$ is nonzero and of unique degree.
	\end{Theorem}

    The Teichmüller--Nakayama form is the analogue of the Smith form for the skew-polynomial ring $\mathcal M[D]$. If $R(D)$ has constant coefficients and the transformations are restricted to $\R[D]$-unimodular matrices, one obtains the Smith form instead. The Teichmüller--Nakayama form is simpler because it permits the larger class of $\mathcal M[D]$-unimodular transformations; see~\cite[Remark~2(i)]{part1}.

    \begin{Remark} \label{rem:tnnf}
        For this remark, let $L \in \{-1, 0, \dots, \infty\}$ be arbitrary but fixed.
        \begin{enumerate}[label=(\roman*)]
            \item \label{rem_(i)} If $A(D)\in\MDmn$ is left invertible, then
            \[
                \ker^{L, \infty, \infty}_\Intp A(\ddt) = 0.
            \]
            Here,~$0$ denotes the \emph{zero sheaf} which only consists of the sections constantly zero.
            \item \label{rem_(ii)} If $A(D)\in\MDmn$ is left invertible and $B(D)\in\MD{}{n\times p}$, then
            \[
                \ker^{L, \infty, \infty}_\Int A(\ddt)B(\ddt) = \ker^{L, \infty, \infty}_\Int B(\ddt).
            \]
            \item \label{rem_(iii)} If $R(D)\in\MDmn$ is factorized as in~\eqref{eq:tnnf}, and
            \[
                S(D)= \left[
        		\begin{array}{@{}c|c@{}}
        			\begin{matrix}
        				I_{\ell-1}&\\&r(D)
        			\end{matrix}& 0_{\ell\times (n-\ell)}
        		\end{array}\right],
            \]
            then
            \[
                \ker^{L, \infty, \infty}_\Int R(\ddt) = \ker^{L, \infty, \infty}_\Int S(\ddt)V^{-1}(\ddt).
            \]
            This follows from \ref{rem_(ii)}.
            \item \label{rem_(iv)} If $R(D) \in \MDmn$ is factorized as in \eqref{eq:tnnf} and $R(D)$ is left invertible or right invertible, then $r(D)$ must be a polynomial of degree zero, i.e., $r(D) = r \in \M\setminus\{0\}$. We demonstrate this for the case where $R(D)$ is right invertible.

            Consider the matrix
            \[
            S(D) \coloneqq \begin{bmatrix}
                I_{\ell-1} & & \\
                & r(D) & \\
                & & 0_{(m-\ell)\times(n-\ell)}
            \end{bmatrix}.
            \]
            Since $R(D)$ is right invertible, $S(D)$ must also be right invertible. The right inverse of $S(D)$ has an entry $b(D) \in \M[D]$ such that $r(D)b(D) = 1$. In $\M[D]$, the polynomial degree function satisfies
            \[
            \deg r(D) + \deg b(D) = \deg (r(D)b(D)) = \deg 1 = 0.
            \]
            Thus, $\deg r(D) \leq 0$, and with $r(D) \neq 0$, we conclude that $r(D)$ is of degree zero. \qedhere
        \end{enumerate}
    \end{Remark}

    \begin{Lemma}\label{lemma:cont_degree}
        Let $R(D) \in \M[D]^{m \times n}$ be right invertible. Then, the local behaviour $\ker^{L, \infty, \infty}_\Int R(\ddt)$ satisfies the \hyperref[def:continuation]{continuation property} (Definition~\ref{def:continuation}) for all $L \geq -1$.
    \end{Lemma}
    \begin{proof}
        Fix $[a, b] \in \Int$ with $a < b$ and define the $\Int$-presheaf
        \[
            \mathcal{S}_R^L \coloneqq \ext_{L} \left( \ker^{\infty}_\Intp R(\ddt) \right), \qquad L \geq -1.
        \]
        We show that any $s \in \mathcal{S}_R^\infty[a, b]$ can be smoothly ($\Ca^\infty$) extended to a solution $\hat{s} \in \mathcal{S}_R^\infty[{a -\varepsilon}, b +\varepsilon]$, for some $\varepsilon > 0$. Then, by Proposition~\ref{prop:presheaves_ftw}, the \hyperref[def:continuation]{continuation property} follows for $\mathcal{S}_R^L$ and
        \[
            \ker^{L, \infty, \infty}_\Int R(\ddt) = \left( \mathcal{S}_R^L \right)^\# \qquad \text{for all $L \geq -1$}.
        \]

        Fix $s \in \mathcal{S}_R^\infty[a, b]$ and let $R(D) = U(D)^{-1} S(D) V(D)^{-1}$ be factorized as in \eqref{eq:tnnf}, where
        \[
            S(D) \coloneqq \left[
    		\begin{array}{@{}c|c@{}}
			    \begin{matrix}
				    I_{m-1}&\\&r(D)
			    \end{matrix}& 0_{m\times (n-m)}
		      \end{array}\right].
        \]
        Here, we used that $R(D)$ has full row rank. By Remark~\ref{rem:tnnf}~\ref{rem_(iv)}, $r(D)$ is a nonzero real meromorphic function $r(D) = r \in \M \setminus\{0\}$. Hence, $S(D)$ is a right invertible real meromorphic matrix $S(D) = S \in \M^{m \times n}$.
        By Remark~\ref{rem:tnnf}~\ref{rem_(iii)}, $s$ satisfies
        \[
            S V(\ddt)^{-1} s = 0.
        \]
        Defining $w \coloneqq V(\ddt)^{-1} s \in \left(\CIntp{-1, -1, \infty}\right)^n[a, b]$ yields ${S w = 0 }$. Our aim is to extend $w$ to a solution $\hat{w} \in \left(\CIntp{-1, -1, \infty}\right)^n[a - \varepsilon, b + \varepsilon]$ satisfying $S\hat{w} = 0$, in such a way that
        \[
            \hat{s} \coloneqq V(\ddt) \hat{w} \in \mathcal{S}_R^\infty[a - \varepsilon, b + \varepsilon],
        \]
        i.e., $\hat{s}$ is a smooth extension of $s$.

        From $S w = 0$, we know that for $w = (w_1, \dots, w_n)$, it follows that $w_1 = \dots = w_{m-1} = w_m = 0$, and $w_{m+1}, \dots, w_n$ are free. We construct the continuation $\hat{w} = (\hat{w}_1, \dots, \hat{w}_n)$ of $w$ as follows: Clearly, one must choose $\hat{w}_1 = \dots = \hat{w}_m = 0$ so that $S \hat w = 0$ is still satisfied. Since $\hat{w}_{m+1}, \dots, \hat{w}_n$ are free, the remaining challenge for the continuation is to ensure that $\hat{s} \coloneqq V(\ddt) \hat{w}$ is smooth in $a$ and $b$. This is done by determining the orders $n_a$ and $n_b$ of the poles at $a$ and $b$ for each~$w_j$, where $j = m+1, \dots, n$. A smooth continuation of
        \[
            t \mapsto w_j(t) (t-a)^{n_a} (t-b)^{n_b}
        \]
        can then be chosen, and dividing by $(t-a)^{n_a} (t-b)^{n_b}$ yields the extension $\hat{w}_j$.

        We show that this extension is suitable. First, note that by construction, the only singularities of $\hat w$ in $[a - \varepsilon, b + \varepsilon]$ are poles, because the singularities in $[a, b]$ cannot arise from differentiation but only from multiplication with the meromorphic entries of $V(\ddt)^{-1}$. By construction of our extension, this remains true in the extended interval $[a - \varepsilon, b + \varepsilon]$. In particular, $\hat{w}$ can be written as $\hat{w} = g \cdot \hat{h}$, where $g \in \M$ and $\hat{h} \in \left(\Ca^{\infty}_\Intp\right)^n[a-\varepsilon, b+\varepsilon]$.

        It follows that there exists $P(D) \in \MD{}{n \times n}$ such that
        \[
            \hat{s} \coloneqq V(\ddt) \hat{w} = V(\ddt) (g \cdot \hat{h}) = P(\ddt) \hat{h}.
        \]
        The matrix polynomial $P(D)$ is obtained by applying the product rule to $V(\ddt) (g \cdot \hat{h})$ and then sorting terms according to the order of derivatives of $\hat{h}$.

        Since $P(\ddt)$ has meromorphic coefficients, it maps the smooth function $\hat{h}$ to a function $\hat{s}$ whose singularities are at most poles. Since $\hat{s}\rest{[a,b]}=s$ is smooth on $[a,b]$, $\hat{s}$ has no poles there. Thus, $\hat{s}$ is smooth on $[a-\delta,b+\delta]$ for some $0<\delta\leq\varepsilon$.

        It remains to verify that
        \[
            R(\ddt) \hat{s} \rest{[a-\delta, b+\delta]} = 0.
        \]
        We compute
        \[
            R(\ddt) \hat{s} = U(\ddt)^{-1} S V(\ddt)^{-1} V(\ddt) \hat{w} = U(\ddt)^{-1} S \hat{w}.
        \]
        By the premise, we have $S\hat{w} = 0$, which implies $R(\ddt) \hat{s} = 0$.
    \end{proof}

    Finally, the algebraic criterion for controllability is established in Theorem~\ref{thm:main}. Its close correspondence with the results of \citet[Theorems~6,~7]{zerz} and \citet[Theorem 3.2]{part1} justifies the proposed Definition~\ref{def:contr_new} of controllability within this framework.

    The proof adapts the strategy of \cite[Theorem~3.2]{part1}. In Subsection~\ref{subsec:global_comp}, we discussed two structural issues concerning the local solution spaces. In particular, the claimed equivalence between controllability and controllability to zero relies on a vector space structure that is not available as stated. The proof of \cite[Theorem~3.2]{part1} also contains minor inaccuracies in its interval choices. In the necessity direction, the transition to zero is not shown to be completed within the interval $\I$ on which the analyticity argument is carried out. In the sufficiency direction, the constructed trajectory is defined only on the transition interval $\I=(\tau_0,\tau_1)$, but is asserted to agree with the original trajectory for $t\leq\tau_0$ and with zero for $t\geq\tau_1$. The proof below corrects these issues.

    \begin{Theorem}\label{thm:main}
        Let $R(D) \in \MDmn$ have full row rank, $V^{-1}(D) \in \MD{}{n\times n}$ and $r(D) \in \M[D]\setminus\{0\}$ be as in \eqref{eq:tnnf}, and $L \geq \deg V^{-1}(D) + \deg r(D) - 1$. Then, the local behaviour $\ker^{L, \infty, \infty}_\Int R(\ddt)$ is controllable if and only if $R(D)$ is right invertible.
    \end{Theorem}

    \begin{proof}
        Fix $L \geq \deg V^{-1}(D) + \deg r(D) - 1$. To simplify our notation, we write
        \[
            \mathfrak{B}_A \coloneqq \ker^{L, \infty, \infty}_\Int A(\ddt)
        \]
        for all matrix polynomials $A(D) \in \MD{}{o \times p}$ appearing in this proof.

        Let $R(D) = U(D)^{-1} S(D) V(D)^{-1}$ be factorized as in \eqref{eq:tnnf}, where
        \[
            S(D) \coloneqq \left[
    		\begin{array}{@{}c|c@{}}
			    \begin{matrix}
				    I_{m-1}&\\&r(D)
			    \end{matrix}& 0_{m\times (n-m)}
		      \end{array}\right].
        \]
        Here, we used that $R(D)$ has full row rank. By Remark~\ref{rem:tnnf}~\ref{rem_(iii)}, $\mathfrak{B}_R = \mathfrak{B}_{SV^{-1}}$.

        By Remark~\ref{rem:tnnf}~\ref{rem_(iv)}, $R(D)$ is right invertible if and only if $r(D)$ is a polynomial of degree~$0$, or equivalently, $r(D)$ is a nonzero real meromorphic function. It remains to show that $\mathfrak{B}_R$ is controllable if and only if~$r(D)\in\M\setminus\{0\}$. In order to prove this, let $\T$ be defined as the union of all poles of the entries of the coefficient matrices of $r(D)$, $V(D)$ and $V(D)^{-1}$ together with the zeros of the leading coefficient of $r(D)$. Since $\T$ is a finite union of discrete sets, it is also discrete. \medskip \\
        \noindent $\Longrightarrow:$ Seeking a contradiction, assume that $r(D)\in\MD{}{}$ is a polynomial of degree at least~$1$. By the theory of linear ordinary differential equations, for all $[a, b]\in\Intp$ with $[a, b]\cap\T = \emptyset$ and any given initial value $\varphi_t \in \CInt{\deg r(D) - 1}[t, t] = \{t\} \times \R^{\deg r(D)}$ at $t \in [a, b]$, there exists a unique solution $\varphi\in\mathfrak{B}_{r}[a, b]$ satisfying the initial value problem
    	\[
    		r(\ddt)\varphi = 0, \qquad \varphi\rest{[t, t]} = \varphi_t,
    	\]
    	where the restriction $\varphi\rest{[t, t]}$ is evaluated in $\CInt{\deg r(D) - 1}$. In particular, with $\varphi_t \neq 0$, we obtain a nonzero solution $\varphi \in \left(\ker_\Int^{\infty} r(\ddt) \right)[a, b] \setminus \{0\}$. Note that $\varphi$ is actually smooth, i.e., $\varphi \in \CInt{\infty}[a, b]$.

    	Define
        \[
            f \coloneqq V(\ddt) \hspace{1pt}\varphi \hspace{1pt}\mathbf{e}_m,
        \]
        where $\mathbf{e}_m\in\R^n$ denotes the $m$th canonical basis vector of $\R^n$. By definition of $\T$, and since $\varphi$ is nonzero and smooth, this yields $f \in \mathfrak{B}_{R}[a, b]\setminus\{0\}$ which is also smooth.

        Now, suppose that $\mathfrak{B}_R$ is controllable. Choose $c$ and $d$ such that $b < c < d$ and $[a, d] \cap \T = \emptyset$. By assumption, $f$ is controllable to $0 \in \mathfrak{B}_R[c, d]$, i.e., there exists $g \in \mathfrak{B}_R[b, c]$ compatible with $f$ and $0$, and the glued function is $e \coloneqq f \cup g \cup 0 \in \mathfrak{B}_R[a, d]$.
        Therefore,
        \[
            S(\ddt) V(\ddt)^{-1} e = 0,
        \]
        which means
        \[
            V(\ddt)^{-1} e = (\underbrace{0,\dots,0}_{m-1 \text{ times}}, \varphi_m,\dots,\varphi_n),
        \]
        for some $\varphi_m,\dots, \varphi_n\in\CIntp{\deg r(D) - 1, \infty, \infty}[a, d]$. By existence and uniqueness of solutions, $\varphi_m$ is actually smooth and $\varphi_m \in \mathfrak{B}_r[a, d]$. By construction,
        \begin{equation}\label{eq:thm:main:1}
            \varphi_m \rest{[a, b]} = \varphi \qquad \text{and} \qquad \varphi_m\rest{[c, d]} = 0.
        \end{equation}
        Since we have uniqueness of solutions in $[a, d]$, we obtain
        \[
            \varphi_m = 0 \in \mathfrak{B}_r[a, d] \qquad \Longrightarrow \qquad \varphi = 0 \in \mathfrak{B}_r[a, b].
        \]
        This contradicts $\varphi\in\mathfrak{B}_{r}[a, b]\setminus\{0\}$. \medskip \\
        \noindent $\Longleftarrow$: Let $r(D)\in\M\setminus\{0\}$. First, we show that $\mathfrak{B}_R$ is controllable to $0$. For this, let $a < b < c < d$ and~$f\in\mathfrak{B}_R[a, b]$. By Lemma~\ref{lemma:cont_degree}, $\mathfrak{B}_R$ satisfies the \hyperref[def:continuation]{continuation property}. Hence, there exists $w^0 \in \mathfrak{B}_R[a, b + 2\varepsilon]$ with $w^0\rest{[a, b]} = f$. Here, $\varepsilon > 0$ is chosen sufficiently small such that
        \[
            b + 2\varepsilon < c, \quad  (b, b + 2\varepsilon] \cap \T = \emptyset, \quad \text{and} \quad w^0\rest{[b, b+2\varepsilon]} \in \CIntp{\infty}[b, b+2\varepsilon].
        \]
        Now, choose a function $\chi\in\Cinf(\R)$ such that
        \[
            \chi(t) = \begin{cases}
                1, &t \leq b + \varepsilon,\\
                0, &t \geq b + 2\varepsilon,
            \end{cases}
        \]
        and define $w  \in \left(\CIntp{-1, -1, \infty}\right)^n[a, d]$ by $w \rest{[a, b + 2 \varepsilon]} = V(\ddt) \chi V(\ddt)^{-1} w^0$ and $w \rest{[b + 2\varepsilon, d]} = 0$. We have
        \[\arraycolsep=2pt\begin{array}{rcl}

            S(\ddt)V(\ddt)^{-1} w\rest{[a, b + 2\varepsilon]} &=& S(\ddt)V(\ddt)^{-1} V(\ddt) \chi V(\ddt)^{-1} w^0 \ab
            &=& S(\ddt) \chi V(\ddt)^{-1} w^0 \ab
            &=& \chi S(\ddt) V(\ddt)^{-1} w^0 \ab
            &=& 0,
        \end{array}\]
        where $S(\ddt) \chi = \chi S(\ddt)$ because $S(D)\in\M^{m\times n}$. Therefore, $S(\ddt)V(\ddt)^{-1} w = 0$. By construction, $w$ is of regularity $\CIntp{L, \infty, \infty}$
        \[\begin{array}{rl}
            \text{on } [a, b + \varepsilon) & \text{because $w$ is equal to $w^0$ on $[a, b + \varepsilon)$,} \\
            \text{on } (b, b + 2\varepsilon] & \text{because $V(\ddt)$, $\chi$, $V(\ddt)^{-1}$, and $w^0$ are smooth on $(b, b + 2\varepsilon]$,} \\
            \text{on } (b + 2\varepsilon, d] & \text{because $w$ is equal to $0$ on $(b + 2\varepsilon, d]$.}
        \end{array}\]
        The connection at $b+2\varepsilon$ is smooth, since $\chi$ is flat there. It follows that $w \in \left(\CIntp{L, \infty, \infty}\right)^n[a, d]$, i.e., $w\in \mathfrak{B}_{S V^{-1}}[a, d] = \mathfrak{B}_R[a, d]$. Setting $g \coloneqq w\rest{[b, c]} \in \mathfrak{B}_R[b, c]$ yields that $f$, $g$, and $0 \in \mathfrak{B}_R[c, d]$ are compatible.

        We have shown that $\mathfrak{B}_R$ is controllable to $0$. By Lemma~\ref{lemma:cont_degree} and Theorem~\ref{thm:contr}~\ref{thm:contr:(iii)}, it follows that $\mathfrak{B}_R$ is controllable. This completes the proof.
    \end{proof}

    \begin{Remark}
        For fixed $R(D)$, one may ask how controllability of $\ker^{L,\infty,\infty}_{\Int} R(\ddt)$ depends on~$L$. Theorem~\ref{thm:main} shows that for sufficiently large $L$, controllability forces right invertibility of $R(D)$.

        For small values of $L$, controllability holds much more generally. For every $R(D)\in\MDmn$, the local behaviour $\ker^{-1,\infty,\infty}_{\Int} R(\ddt)$ is controllable. Another example is $r(D)=D^2$, where the smooth local solutions are affine linear functions. Any two solution pieces can be connected continuously by an affine linear function on the intermediate interval. Therefore, $\ker^{0,\infty,\infty}_{\Int} r(\ddt)$ is controllable although $r(D)$ is not right invertible.

        This suggests that there is a threshold in $L$: for small $L$, controllability holds generally, whereas for large $L$, controllability imposes algebraic restrictions on $R(D)$. In the scalar case, the example $r(D)=D^2$ together with Theorem~\ref{thm:main} suggests that the threshold is given by $L=\deg r(D)-2$, where $\deg r(D)\geq 1$. We keep this formulation deliberately cautious, since we do not provide a proof in this paper. Doing so would raise further questions, in particular concerning analogous bounds for matrix polynomials $R(D)$, and would take the discussion beyond the scope of the present paper.
    \end{Remark}

    \begin{Remark}
        The lower bound on $L$ in Theorem~\ref{thm:main} depends on the chosen Teichmüller--Nakayama factorization $R(D)=U(D)^{-1}S(D)V(D)^{-1}$ since $\deg V(D)^{-1}$ need not be uniquely determined by $R(D)$. One could define
        \[
            \nu(R(D)) \coloneqq
            \min \bigset{
                \deg V(D)^{-1} }{
                R(D)=U(D)^{-1}S(D)V(D)^{-1}
                \text{ as in \eqref{eq:tnnf}} },
        \]
        and then restate Theorem~\ref{thm:main} with the optimal condition
        \[
            L \geq \nu(R(D)) + \deg r(D) - 1.
        \]
        Explicit bounds for the degrees of the transformation matrices are known in special cases. In particular, for a full-rank matrix $R(D)\in\R(t)[D]^{n\times n}$ with rational-function coefficients, \cite[Theorem~4.1~(b)]{giesbrecht_heinle} and \cite[Corollary~3.3]{giesbrecht_kim} imply
        \[
            \nu(R(D))\leq n(n-1)\deg R(D).
        \]
        However, we are not aware of a corresponding explicit bound for general full-row-rank matrices $R(D)\in\M[D]^{m\times n}$ in the form considered here.
    \end{Remark}

    By choosing $L=\infty$, the lower bound in Theorem~\ref{thm:main} is satisfied independently of the chosen factorization. Therefore, we obtain as an immediate corollary a more direct translation of the controllability criteria of \cite[Theorem~3.2]{part1} and \cite[Theorems~6,~7]{zerz} in the present local setting.

    \begin{Corollary}\label{cor:main}
        Let $R(D)\in\MDmn$ have full row rank. Then, the local behaviour $\ker^{\infty,\infty,\infty}_{\Int} R(\ddt)$ is controllable if and only if $R(D)$ is right invertible.
    \end{Corollary}

\appendix
\setcounter{section}{1}
\setcounter{equation}{0}
\setcounter{theorem}{0}
\renewcommand{\thesection}{\Alph{section}}
\renewcommand{\thesubsection}{\thesection.\arabic{subsection}}

\section*{Appendix}\label{app}

    \noindent This appendix has two purposes. First, we verify that the definition of a $\Intpp$-sheaf given in the \hyperref[sec:intro]{introduction}, based on \hyperref[prop]{Properties}~\ref{prop:i}--\ref{prop:v} and Subsection~\ref{subsec:Int_comp}, agrees with the standard definition of a sheaf. In particular, we show that requiring unique gluing of two compatible sections, as done in the main text, is equivalent to requiring unique gluing of finitely many compatible sections. Second, we present the deferred proof of Theorem~\ref{thm:Int&Intp}.

    We begin by providing the general definition of a $\Intpp$-sheaf. The definition of a sheaf depends on which families of sections are regarded as \eanf{overlapping}. To formalize this, we use the notion of a \emph{coverage} introduced by \citet[Part~C, Definition~2.1.1]{elephant}. A coverage generalizes the concept of open coverings in topological spaces. Roughly speaking, a coverage on $\Intpp$ assigns to every object $[a,b]$ a collection of so-called \emph{covering families} $\{[a_i,b_i]\subseteq [a,b]\}_{i\in I}$ subject to a consistency requirement: for every $[x,y]\subseteq [a,b]$ in $\Intpp$, a covering family can be pulled back to a covering family of $[x,y]$. In $\Intp$, this pullback is given by
    \begin{equation}\label{eq:pullback_Intp}
        \bigset{[a_i, b_i] \cap [x, y] \subseteq [x, y]}{i \in I \land (a_i, b_i) \cap (x, y) \neq \emptyset}
    \end{equation}
    and in $\Int$ as
    \begin{equation}\label{eq:pullback_Int}
        \bigset{[a_i, b_i] \cap [x, y] \subseteq [x, y]}{i \in I \land [a_i, b_i] \cap [x, y] \neq \emptyset}.
    \end{equation}

    \begin{Definition}
        For $k \in \{1, 2, \dots, \infty\}$, define the coverages $\mathcal{J}_{<k}^\mathrm{p}$ and $\mathcal{J}_{<k}$ on $\Intp$ and $\Int$, respectively, by setting
        \[
            \mathcal{J}_{<k}^\mathrm{p}([a, b]) \coloneqq \bigset{ \{[a_i, b_i] \subseteq [a, b]\}_{i \in I}}{\bigcup_{i \in I} (a_i, b_i) = (a, b), \text{ where } \abs{I} < k} \quad \text{for } [a, b] \in \Intp
        \]
        and
        \[
            \mathcal{J}_{<k}([a, b]) \coloneqq \bigset{ \{[a_i, b_i] \subseteq [a, b]\}_{i \in I}}{\bigcup_{i \in I} [a_i, b_i] = [a, b], \text{ where } \abs{I} < k} \quad \text{for } [a, b] \in \Int. \qedhere
        \]
    \end{Definition}

    To verify that $\mathcal{J}_{<k}^{\mathrm{(p)}}$ defines a coverage on $\Intpp$, it suffices to check that the pullbacks~\eqref{eq:pullback_Intp} and~\eqref{eq:pullback_Int} are covering families and that their size is less than $k$.

    \begin{Definition}[Sheaf, {cf.\ \cite[Part~C, Definition~2.1.2]{elephant}, \cite[\href{https://stacks.math.columbia.edu/tag/00YR}{Tag 00YR}]{stacks-project}}]\label{def:sheaf_on_site}
        ~\\
        \indent Let $\mathcal{F}\colon\Intpp\op \to \Vect$ be a $\Intpp$-presheaf. Given a family $\{[a_i, b_i] \subseteq [a, b]\}_{i \in I}$ in $\Intpp$, a family of elements $s_i \in \mathcal{F}[a_i, b_i]$, $i \in I$, is called \emph{compatible} if and only if for all $i, j \in I$ and $[x, y] \subseteq [a_i, b_i] \cap [a_j, b_j]$ in $\Intpp$, we have $s_i\rest{[x, y]} = s_j\rest{[x, y]}$.

        Let $\mathcal{J}$ be a coverage on $\Intpp$. $\mathcal{F}$ is called a \emph{$\mathcal{J}$-sheaf} if and only if for all $\mathcal{J}$-covering families $\{[a_i, b_i] \subseteq [a, b]\}_{i \in I}$ and all compatible families $s_i \in \mathcal{F}[a_i, b_i]$, $i \in I$, there exists a unique $s \in \mathcal{F}[a, b]$ such that $s\rest{[a_i, b_i]} = s_i$ for each $i \in I$.

        If the coverage $\mathcal{J}$ on the category $\Intpp$ is apparent, we may simply write \emph{$\Intpp$-sheaf} instead of $\mathcal{J}$-sheaf.
    \end{Definition}

    In the main part of this work, the $\Intpp$-sheaves considered are $\mathcal{J}_{<3}^{\mathrm{(p)}}$-sheaves, that is, sheaves admitting unique gluing of two (or fewer) compatible sections. Thus, the use of the term \eanf{sheaf} in the main text is justified with respect to the coverage $\mathcal{J}_{<3}^{\mathrm{(p)}}$.

    Arguably, a more useful notion is that of a $\mathcal{J}_{<\infty}^{\mathrm{(p)}}$-sheaf, which allows unique gluing of any finite family of compatible sections. In the following theorem, we prove that these notions are equivalent by induction. Intuitively, the argument is straightforward: if one can glue $k$ (or fewer) compatible sections, then one can also glue $k+1$ compatible sections by first gluing the first $k$ sections and then gluing the resulting section with the remaining one.

    \begin{Proposition}\label{prop:cov:equiv}
        Let $k < \ell$ in $\{3, 4, \dots, \infty\}$ and $\B$ be a presheaf on $\Intpp$. Then, $\B$ is a $\mathcal{J}_{<k}^\mathrm{(p)}$-sheaf if and only if it is a $\mathcal{J}_{<\ell}^\mathrm{(p)}$-sheaf.
    \end{Proposition}
    \begin{proof}
        If $\B$ is a $\mathcal{J}_{<\ell}^\mathrm{(p)}$-sheaf, then it is also a $\mathcal{J}_{<k}^\mathrm{(p)}$-sheaf, since for all $[a, b] \in \Intpp$,
        \begin{equation*}
            \mathcal{J}_{<k}^\mathrm{(p)}([a, b]) \subseteq \mathcal{J}_{<\ell}^\mathrm{(p)}([a, b]) \quad \Longleftrightarrow \quad k \leq \ell.
        \end{equation*}

        Now, suppose that $\B$ is a $\mathcal{J}_{<k}^\mathrm{(p)}$-sheaf. We show that it is also a $\mathcal{J}_{<k+1}^\mathrm{(p)}$-sheaf. To this end, consider a compatible family of sections $s_i \in \B[a_i, b_i]$ defined on a covering family
        \[
            \{[a_i, b_i] \subseteq [a, b] \mid i = 1, \dots, k\} \in \mathcal{J}_{<k+1}^\mathrm{(p)}([a, b])
        \]
        of $[a, b] \in \Intpp$. Without loss of generality, let $a_k = \max \bigset{a_i}{i=1, \dots, k}$, so that $\bigcup_{i<k}[a_i, b_i]$ is still an interval. Then, the sections $s_1, \ldots, s_{k-1}$ are uniquely glueable according to the sheaf axiom of $\B$ with respect to $\mathcal{J}_{<k}^\mathrm{(p)}$. This yields a unique section $s_{<k}$ of $\B$ on the interval $\bigcup_{i<k}[a_i, b_i]$. Now, $s_{<k}$ and $s_k$ are also compatible. Hence, they can be glued, yielding a unique section $s_{<k+1}$ on $\bigcup_{i<k+1}[a_i, b_i]$. This verifies the sheaf axiom for $\mathcal{J}_{<k+1}^\mathrm{(p)}$, i.e., $\B$ is a $\mathcal{J}_{<k+1}^\mathrm{(p)}$-sheaf.

        For finite $\ell$, the claim follows by induction. For $\ell = \infty$, it follows from
        \begin{equation*}
            \mathcal{J}_{<\infty}^\mathrm{(p)}([a, b]) = \bigcup_{k \leq m < \infty} \mathcal{J}_{<m}^\mathrm{(p)}([a, b]) \qquad \text{for all $[a, b] \in \Intpp$}. \qedhere
        \end{equation*}
    \end{proof}

    We proceed with the proof of Theorem~\ref{thm:Int&Intp}. It relies on categorical (co-)limit constructions, whose definitions are omitted; for background, see \cite{MacLane:CatsFTW} or \cite{awodey2010category}.

    \begin{proof}[Proof of Theorem~\ref{thm:Int&Intp}]
        An $\Int$-sheaf restricts to a $\Intp$-sheaf by simply restricting the underlying coverage. Moreover, for every $a < b < c$, the sheaf axiom applied to the covering
        \[
            [a,c] = [a,b] \cup [b,c]
        \]
        implies that the map
        \[
            \B[a,c] \longrightarrow \B[a,b] \oplus \B[b,c],
            \qquad
            f \longmapsto
            \left(f\rest{[a,b]},f\rest{[b,c]}\right),
        \]
        is injective.

        Conversely, let $\B$ be a $\Intp$-sheaf of vector spaces satisfying property~\eqref{eq:sheaf_ext}. Fix $b \in \R$ and set
        \[
            \B_{b-} \coloneqq \displaystyle\varinjlim_{a < b} \B[a,b],\qquad
            \B_{b+} \coloneqq \displaystyle\varinjlim_{c > b} \B[b,c],\qquad
            \B_b    \coloneqq \displaystyle\varinjlim_{a < b < c} \B[a,c].
        \]
        Conceptually, the half-stalks $\B_{b-}$ and $\B_{b+}$ consist of the left and right halves of germs. However, these halves may or may not be restrictions of complete germs originating from the stalk $\B_b$. Hence, there are restriction maps $\B_{b} \to \B_{b-}$ and $\B_{b} \to \B_{b+}$, which need not be injective or surjective.

        The induced map, however,
        \[
            \B_b \longrightarrow \B_{b-} \oplus \B_{b+},
            \qquad
            f \longmapsto
            \left(f\rest{b-},f\rest{b+}\right),
        \]
        is injective. Indeed, suppose that $\left(f\rest{b-},f\rest{b+}\right) = 0$. After choosing a representative of $f \in \B_b$ and shrinking its domain if necessary, both of its restrictions to the left and right half-intervals vanish. Property~\eqref{eq:sheaf_ext} then implies that the representative itself vanishes, and hence $f=0$.

        Define $\B[b, b]$ as the pushout of the restriction maps $\B_b \to \B_{b-}$ and $\B_b \to \B_{b+}$ in the category of real vector spaces. That is, $\B[b, b]$ is a real vector space equipped with morphisms
        \[
            \B_{b-} \to \B[b, b] \quad \text{and} \quad \B_{b+} \to \B[b, b]
        \]
        such that the square
        \[
            \begin{tikzcd}[column sep=small]
                \B_b \arrow[r] \arrow[d] & \B_{b+} \arrow[d] \\
                \B_{b-} \arrow[r] & \B[b, b]
            \end{tikzcd}
        \]
        commutes, and such that the following universal property holds: for every such vector space $V$ equipped with two morphisms $\B_{b-} \to V$ and $\B_{b+} \to V$ for which the compositions $\B_b \to \B_{b-} \to V$ and $\B_b \to \B_{b+} \to V$ are equal, there exists a unique morphism $\B[b, b] \to V$ making the diagram
        \[
            \begin{tikzcd}[column sep=small]
                \B_b \arrow[r] \arrow[d] & \B_{b+} \arrow[d] \arrow[ddr, bend left] \\
                \B_{b-} \arrow[r] \arrow[drr, bend right] & \B[b, b] \arrow[dashed, dr] \\
                & & V
            \end{tikzcd}
        \]
        commute.
        The pushout is unique up to isomorphism and can be constructed as the quotient
        \[
            \B[b, b] \coloneqq \left( \B_{b-} \oplus \B_{b+} \right) \big/ \bigset{ \left( f\rest{b-}, -f\rest{b+} \right) }{ f \in \B_b }.
        \]
        The maps $\B_{b-} \to \B[b,b]$ and $\B_{b+} \to \B[b,b]$ are obtained by composing the inclusions into $\B_{b-}\oplus\B_{b+}$ with the quotient map onto $\B[b,b]$.

        The commutative diagram
        \[\begin{tikzcd}[column sep=small, row sep=small]
            & \B[a, c] \arrow[dl] \arrow[dd] \arrow[dr] & \\
            \B[a, b] \arrow[dd] & & \B[b, c] \arrow[dd]\\
            & \B_b \arrow[dl] \arrow[dr] & \\
            \B_{b-} \arrow[dr] & & \B_{b+} \arrow[dl]\\
            & \B[b, b]
        \end{tikzcd}\]
        visualizes the relations between all vector spaces.

        Intuitively, $\B[b, b]$ contains exactly the information that $\B_{b-}$ and $\B_{b+}$ have in common via their restrictions from $\B_b$. In other words, $\B_b$ now represents pairs of sections from $\B_{b-}$ and $\B_{b+}$ that coincide on $\B[b, b]$, i.e., it is the pullback of $\B_{b-} \to \B[b, b]$ and $\B_{b+} \to \B[b, b]$. This follows from the injectivity of $\B_b \longrightarrow \B_{b-} \oplus \B_{b+}$.

        More rigorously, for $f \in \B[a, b]$ and $g \in \B[b, c]$ with $a < b < c$, we have
        \[\begin{aligned}[b]
                {}&{}
                    \text{$f$ and $g$ are compatible in $\Int$} \ab
                \Longleftrightarrow \quad {}&{}
                    (f\rest{b-}, 0) \sim (0, g\rest{b+}) \ab
                \Longleftrightarrow \quad {}&{}
                    \exists h \in \B_b \colon (f\rest{b-}, 0) = (h\rest{b-}, g\rest{b+} - h\rest{b+}) \ab
                \Longleftrightarrow \quad {}&{}
                    \exists h \in \B_b \colon h\rest{b-} = f\rest{b-} \land h\rest{b+} = g\rest{b+} \ab
                \Longleftrightarrow \quad {}&{}
                    \exists \varepsilon > 0 \exists h \in \B[b-\varepsilon, b+\varepsilon] \colon h\rest{[b-\varepsilon, b]} = f\rest{[b-\varepsilon, b]} \land h\rest{[b, b+\varepsilon]} = g\rest{[b, b+\varepsilon]} \ab
                \Longleftrightarrow \quad {}&{}
                    \exists \varepsilon > 0 \exists h \in \B[b-\varepsilon, b+\varepsilon] \colon \text{$f$, $g$, and $h$ are compatible in $\Intp$.}
        \end{aligned}\]
        Since $\B$ is a $\Intp$-sheaf, $f$, $g$, and $h$ admit a unique gluing
        in $\B[a,c]$. By the assumed injectivity of
        \[
            \B[a,c]\longrightarrow \B[a,b]\oplus\B[b,c],
        \]
        the gluing does not depend on the chosen $h$, and is therefore uniquely determined by $f$ and $g$. Hence, the extended presheaf is an $\Int$-sheaf.
    \end{proof}

    Recall that Example~\ref{ex:counterexample_regular}~\ref{ex:counterexample_regular_(ii)} considers a $\Intp$-sheaf $\B$ that is not extendable via $\ext_\infty$, i.e., using only higher-order derivatives, but can be extended by incorporating additional local data. The following example shows that the construction from Theorem~\ref{thm:Int&Intp} yields precisely this extension.

    \begin{Example}\label{ex:counterexample_regular2}
        Consider again the $\Intp$-sheaf $\B$ from Example~\ref{ex:counterexample_regular}~\ref{ex:counterexample_regular_(ii)}. Following the construction of Theorem~\ref{thm:Int&Intp}, for $a < t < b$ in $\R$, we obtain
        \[\arraycolsep=0pt\begin{array}{rcl}
            \B[t, t] {}&{} \coloneqq {}&{} \bigset {(c_1 f\rest{t-}, c_2f\rest{t+})}{c_1, c_2 \in \R } \big/ \bigset{(c f\rest{t-}, -cf\rest{t+})}{c \in \R } \\
            {}&{} \cong {}&{} \bigset{((c_1 + c_2) f\rest{t-}, 0)}{c_1, c_2 \in \R } \cong \R
        \end{array}\]
        with restriction maps
        \[\arraycolsep=0pt\begin{array}{rlll}
            \B[a, b{}&{}] \to \B[t, t], \qquad c \cdot f\rest{[a, b]} {}&{} \mapsto (c f\rest{t-}, 0{}&{}) \cong c,\\
            \B[a, t{}&{}] \to \B[t, t], \qquad c \cdot f\rest{[a, t]} {}&{} \mapsto (c f\rest{t-}, 0{}&{}) \cong c,\\
            \B[t, b{}&{}] \to \B[t, t], \qquad c \cdot f\rest{[t, b]} {}&{} \mapsto (0, cf\rest{t+} {}&{})\cong c.
        \end{array}\]
        Note that $(g\rest{t-}, 0) = (0, g\rest{t+}) \in \B[t, t]$ for $g \in \B[a, b]$.
    \end{Example}

\begingroup
\makeatletter
\let\addcontentsline\@gobblethree
\section*{Acknowledgements}

I thank Prof.\ Dr.\ Achim Ilchmann and Dr.\ Jonas Kirchhoff for their continued support throughout the development of this work. Numerous discussions and their detailed feedback on successive drafts greatly improved both the content and the structure of the paper. Their careful reading also helped to identify and correct many inaccuracies.

\bibliographystyle{abbrvnat}
\bibliography{Literatur, sources}
\endgroup

\end{document}